\documentclass[reqno,11pt]{amsart}
\usepackage{amsmath,amssymb,amsthm,graphicx,a4wide,enumerate,url}
\usepackage[small,bf]{caption} 
\theoremstyle{plain}
\newtheorem{thm}{Theorem}

\theoremstyle{definition}

\theoremstyle{rem}
\newtheorem{rem}{Rem}

\newcommand{\prn}[1]{\left(#1\right)}

\newcommand{\brk}[1]{\left[#1\right]}

\newcommand{\ud}[1]{\,\mathrm{d}#1}

\newcommand{\vv}{\mathrm{v}}
\newcommand{\vS}{\mathrm{v}_\text{S}} 

\begin{document}
\parskip.9ex

\title[Constructing Set-Valued Fundamental Diagrams from Jamiton Solutions]
{Constructing Set-Valued Fundamental Diagrams from Jamiton Solutions in Second Order Traffic Models}
\author[B. Seibold]{Benjamin Seibold}
\address[Benjamin Seibold]
{Department of Mathematics \\ Temple University \\ \newline
1805 North Broad Street \\ Philadelphia, PA 19122}
\email{seibold@temple.edu}
\urladdr{http://www.math.temple.edu/\~{}seibold}
\author[M. R. Flynn]{Morris R. Flynn}
\address[Morris R. Flynn]
{Department of Mechanical Engineering \\ University of Alberta \\ \newline
Edmonton, AB, T6G 2G8 Canada }
\email{mrflynn@ualberta.ca}
\urladdr{http://websrv.mece.ualberta.ca/mrflynn}
\author[A. R. Kasimov]{Aslan R. Kasimov}
\address[Aslan R. Kasimov]
{4700 King Abdullah University of Science and Technology \\ \newline
Thuwal 23955-6900, Kingdom of Saudi Arabia }
\email{aslan.kasimov@kaust.edu.sa}
\urladdr{http://web.kaust.edu.sa/faculty/aslankasimov}
\author[R. R. Rosales]{Rodolfo Ruben Rosales}
\address[Rodolfo Ruben Rosales]
{Department of Mathematics \\ Massachusetts Institute of Technology \\ \newline
77 Massachusetts Avenue \\ Cambridge, MA 02139}
\email{rrr@math.mit.edu}


\begin{abstract}
Fundamental diagrams of vehicular traffic flow are generally multi-valued in the congested flow regime. We show that such set-valued fundamental diagrams can be constructed systematically from simple second order macroscopic traffic models, such as the classical Payne-Whitham model or the inhomogeneous Aw-Rascle-Zhang model. These second order models possess nonlinear traveling wave solutions, called jamitons, and the multi-valued parts in the fundamental diagram correspond precisely to jamiton-dominated solutions. This study shows that transitions from function-valued to set-valued parts in a fundamental diagram arise naturally in well-known second order models. As a particular consequence, these models intrinsically reproduce traffic phases.
\end{abstract}

\maketitle

\section{Introduction}
\label{sec:introduction}
In this paper a connection between set-valued fundamental diagrams of traffic flow and traveling wave solutions in second order inviscid traffic models is presented.

\subsection{Types of Traffic Models}
A wide variety of mathematical models for vehicular traffic exist. Microscopic models (e.g., \cite{Pipes1953, Newell1961}) describe the individual vehicles and their interactions by ordinary differential equations. Mesoscopic models (e.g., \cite{HermanPrigogine1971, Phillips1979, KlarWegener2000, IllnerKlarMaterne2003}) employ a statistical mechanics perspective in which vehicle interactions are modeled in analogy to kinetic theory. Macroscopic models (e.g., \cite{LighthillWhitham1955, Richards1956, Underwood1961, Payne1971, Payne1979, Lebacque1993, KernerKonhauser1993, KernerKonhauser1994, AwRascle2000}, as well as the discussion below) describe traffic flow using suitable adaptations of the methods of continuum mechanics \cite{Whitham1974}, which yield equations mathematically similar to those from fluid dynamics. There also exist other types of models that are not necessarily based on differential equations, such as probabilistic models (e.g., \cite{AlperovichSopasakis2008}), and cellular models (e.g., \cite{NagelSchreckenberg1992}), which divide the road into segments and prescribe laws for the propagation of vehicles between cells. A detailed overview of the various types of models is given in \cite{Helbing2001}.

All of these types of traffic models are related. For instance, mesoscopic models can be derived from cellular models \cite{NagelSchreckenberg1992} with stochastic components \cite{AlperovichSopasakis2008}; macroscopic models can then be interpreted as limits of mesoscopic models \cite{NelsonSopasakis1999, KlarWegener2000}; cellular models can be interpreted as discretizations (``cell transmission models'') of macroscopic models \cite{Daganzo1994, Daganzo1995_2} and thus relate to kinematic waves \cite{Daganzo2006}; and microscopic models are equivalent to discretizations of macroscopic models in Lagrangian variables. Through these relations, the macroscopic modeling results presented in this work may have consequences for various other types of traffic models.

In this paper, we focus solely on inviscid (aggregated) single-lane macroscopic models for traffic flow on unidirectional roads. These types of models were started with the work of Lighthill \& Whitham \cite{LighthillWhitham1955} and Richards \cite{Richards1956}. The Lighthill-Whitham-Richards (LWR) model describes the temporal evolution of the vehicular density along the road by a scalar hyperbolic conservation law, and is therefore denoted a ``first order model''. The first ``second order model'', due to Payne \cite{Payne1971} and Whitham \cite{Whitham1974}, introduces the vehicular velocity field as an independent variable. As a reaction to certain worrisome mathematical properties of the PW model \cite{Daganzo1995}, Aw \& Rascle \cite{AwRascle2000}, and independently Zhang \cite{Zhang2002} devised a different second order model, which has subsequently been extended and generalized \cite{Greenberg2001, BerthelinDegondDelitalaRascle2008}.

In all of these models, traffic is described by inviscid conservation or balance laws. There also exist macroscopic models with higher spatial derivatives: KdV-type models (e.g., \cite{KurtzeHong1995, KomatsuSasa1995}) involve a dispersive term, and viscous models (e.g., \cite{KernerKonhauser1993, KernerKonhauser1994}) possess a viscosity term. These two classes of models explicitly describe the dynamics of the traffic flow on multiple length scales and avoid discontinuous solutions. In contrast, inviscid models replace thin transition zones (e.g., upstream ends of traffic jams) by shocks, and the specific dynamics inside the transition zone are replaced by appropriate Rankine-Hugoniot jump conditions \cite{Evans1998}.

\subsection{Fundamental Diagram}
\label{subsec:fundamental_diagram}
The fundamental diagram (FD) of traffic flow is the plot of the flow rate (i.e., the number of vehicles passing a specific position on the road per unit time) versus the vehicle density (i.e., the average number of vehicles per unit length). The first FD measurements were done in the 1930s by Greenshields \cite{Greenshields1935}; they led him to propose a quadratic function $Q(\rho) = \rho\,u_\text{max}(1-\rho/\rho_\text{max})$ describing the relation between flow rate and density. Here, $u_\text{max}$ is the velocity that a driver would assume when alone on the road, and $\rho_\text{max}$ is a maximum density at which the traffic flow stagnates (bumper-to-bumper plus a safety distance). This form of the FD was later also used by Richards \cite{Richards1956}. However, subsequent measurements, such as those of Wardrop \& Charlesworth \cite{WardropCharlesworth1954} and H. Greenberg \cite{Greenberg1959}, revealed that a quadratic form is not an accurate description of true FD data, since, for instance, the maximum flow rate generally occurs at densities significantly below $\frac{1}{2}\rho_\text{max}$. In reaction, different types of FD functions were proposed, such as by Lighthill \& Whitham \cite{LighthillWhitham1955}, Greenberg \cite{Greenberg1959}, Underwood \cite{Underwood1961}, Newell \cite{Newell1993}, Daganzo \cite{Daganzo1994}, and Wang \& Papageorgiou \cite{WangPapageorgiou2005}.

However, a more elementary shortcoming of function-based FD models is that true flow rate vs.\ density measurements tend to exhibit a clear systematic spread: for sufficiently large densities, one value of $\rho$ corresponds to multiple values of $Q$. The described ``spreading'' effect can be seen in Fig.~\ref{fig:fd_data}, which displays a plot of flow rate vs.\ density, obtained from single loop sensor measurements at a fixed position on a highway. The data is taken from the RTMC data set \cite{TrafficMnDOT} of 2003, provided by the Minnesota Department of Transportation (Mn/DOT), from a sensor on the southbound direction of I-35W (Minneapolis, MN). Each point in the diagram corresponds to a time interval of $\Delta t = 30\text{s}$. The recorded flow rate equals the number of vehicles that have passed over the sensor during the time interval (the selection of discrete flow rate values is due to the integer nature of the number of vehicles); and the density is obtained from the fraction of time that a vehicle has blocked the sensor. For the plot in Fig.~\ref{fig:fd_data}, we assume that the average vehicle length is $5\text{m}$, and that the possible maximum density is achieved at $1/7.5\text{m}$, i.e., a safety distance of half a vehicle length is the absolute minimum that drivers will accept.

\begin{figure}
\begin{minipage}[b]{.485\textwidth}
\includegraphics[width=\textwidth]{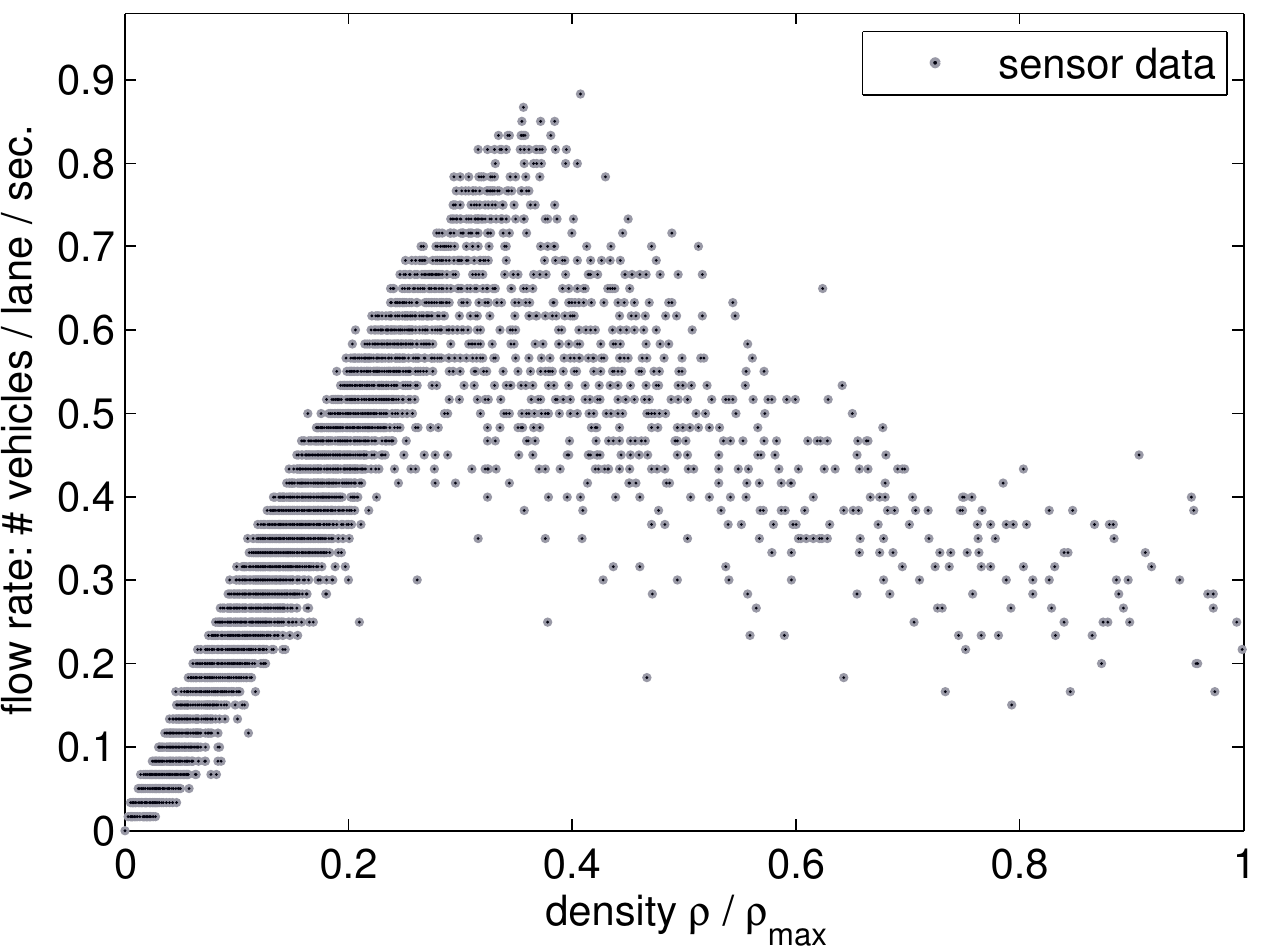}
\vspace{-1.8em}
\caption{Fundamental diagram, obtained from sensor measurement data, aggregated over time intervals of $\Delta t = 30\text{s}$. The horizontal structures are due to the integer nature of the flow rate data.}
\label{fig:fd_data}
\end{minipage}
\hfill
\begin{minipage}[b]{.485\textwidth}
\includegraphics[width=\textwidth]{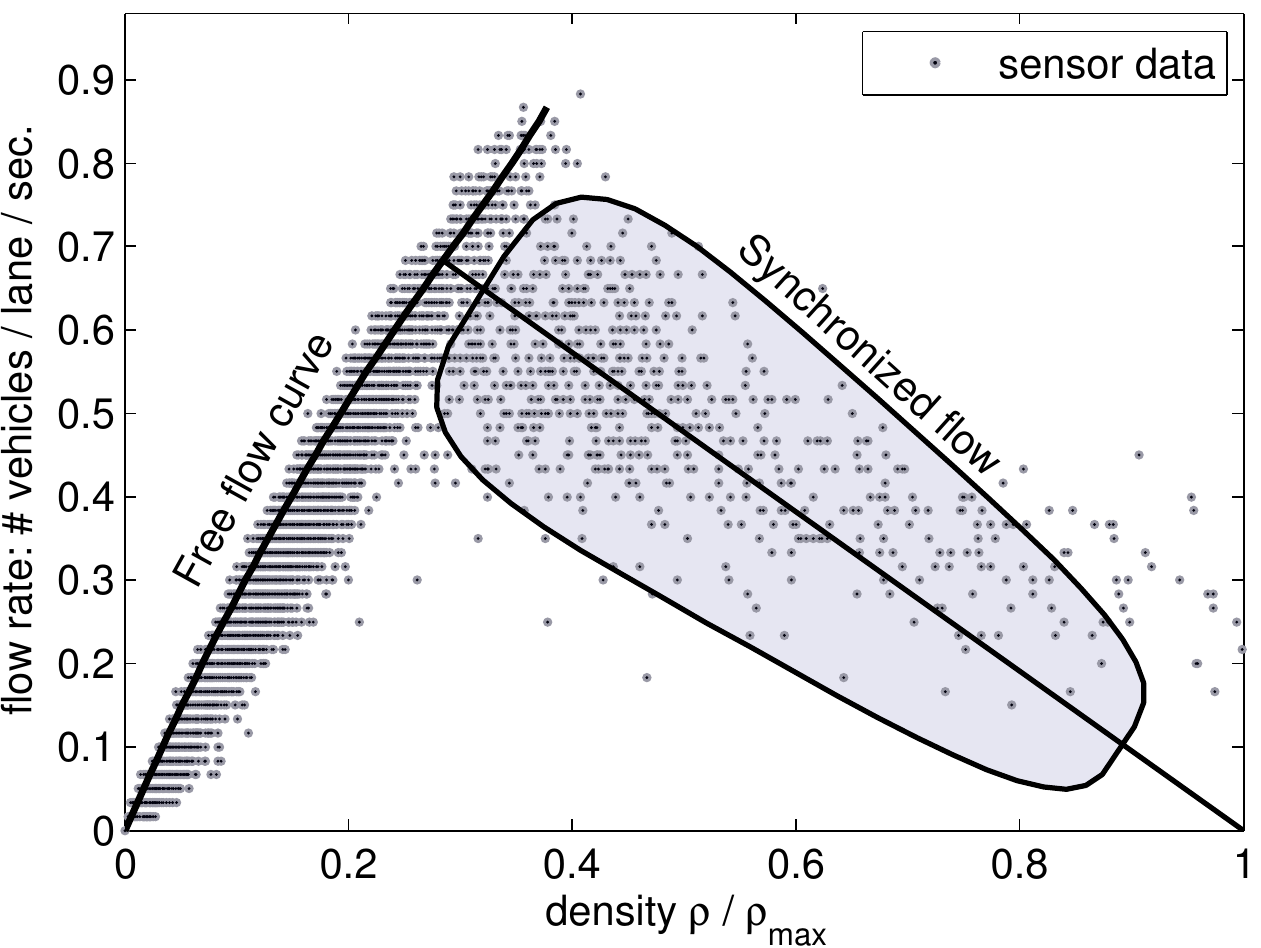}
\vspace{-1.8em}
\caption{Fundamental phase diagram, as frequently depicted in the traffic literature (e.g., \cite{Kerner1998, Helbing2001, Varaiya2005}). Shown here is the free flow curve, and a synchronized flow region.}
\label{fig:fd_data_phase_diagram}
\end{minipage}
\end{figure}

Motivated by the spread of data in the FD, traffic flow is often classified into regimes or \emph{phases}. The two-phase traffic theory (see, e.g., \cite{Helbing2001, Varaiya2005}) divides traffic flow into \emph{free flow} for low densities, and \emph{congested flow} for large densities. Three-phase traffic theory, originally developed by Kerner \cite{Kerner1998}, distinguishes \emph{free flow} for low densities, \emph{synchronized flow} for medium densities, and \emph{wide moving jams} for high densities. A possible depiction of traffic phases for the sensor data is shown in Fig.~\ref{fig:fd_data_phase_diagram}. The free flow region can be characterized by a relatively well defined single valued curve (with the spread due to data noise), while for medium densities (synchronized flow) a multi-valued flow region appears to be needed.

By design, first order traffic models assume a function-based FD, i.e., each density value $\rho$ corresponds to a unique flow rate value $Q$. Therefore, more general traffic models are required to explain the observed spread in FDs. In recent years, a variety of approaches has been presented that describe the spread in FDs by set-valued traffic models. One class is given by phase transition models \cite{Colombo2002, Colombo2003, BlandinWorkGoatinPiccoliBayen2011} that explicitly prescribe one traffic model for low densities, and another traffic model for larger densities, and construct mathematical transition conditions that guarantee well-posed solutions. Other models are based on inserting an instability into a second order model at densities that are known to lead to instability in reality \cite{SiebelMauser2006}.

\subsection{Contribution of this Paper}
In this paper, we demonstrate that the observed structure of FDs, namely a unique $Q$ vs.~$\rho$ relationship for small densities and a set-valued relationship for larger densities (see Fig.~\ref{fig:fd_data}), does already arise within classical second order models, and thus does not necessarily require explicitly prescribed phase transitions. Specifically, we shall develop the theory for two types of models: the classical Payne-Whitham model \cite{Payne1971, Whitham1974}, and the inhomogeneous Aw-Rascle-Zhang model \cite{AwRascle2000, Zhang2002, Greenberg2001}. The key observation is that second order models with a relaxation term possess an intrinsic phase-transition-like behavior: for certain densities, the solutions are very close to the solutions of a first order model (as made precise in the introduction to Sect.~\ref{sec:stability}), while for other densities, the solutions develop \emph{jamitons} \cite{FlynnKasimovNaveRosalesSeibold2009}, i.e., self-sustained, attracting, traveling wave solutions with an embedded shock (described in more detail in Sect.~\ref{sec:jamiton_construction}). As we shall show, these traffic waves result from instabilities and thus can be interpreted as manifestations of \emph{phantom traffic jams}; and the resulting jamitons can be regarded as models for \emph{stop-and-go waves}. Both of these phenomena are observed in real traffic flow \cite{VideoTrafficWaves}, and have been reproduced experimentally \cite{SugiyamaFukuiKikuchiHasebeNakayamaNishinariTadakiYukawa2008}. A key statement of this paper is therefore that set-valued FDs, as well as different traffic phases, are already built into existing second order models of traffic flow, such as the PW model and the inhomogeneous ARZ model. That being said, we emphasize that no statement is made about whether these models are better or worse than other (possibly more complicated) models in describing real vehicular traffic flow.

Beyond the fundamental observation described above, the following theorems are proved in this paper. Based on Whitham's proof of the connection between the linear stability of base state solutions and the sub-characteristic condition (see Thm.~\ref{thm:whitham}), and on the Zel'dovich-von~Neumann-D{\"o}ring theory \cite{FickettDavis1979} of detonations, we show that for both types of second order traffic models, jamitons exist if and only if a corresponding base state solution (i.e., uniform traffic flow) is unstable, see Thms.~\ref{thm:equivalence_jamitons_instability_payne_whitham_model} and~\ref{thm:equivalence_jamitons_instability_aw_rascle_zhang_model}. Moreover, we prove that a jamiton-dominated traffic flow can never possess a higher effective flow rate than a uniform flow of the same average density; this statement is made precise in Thm.~\ref{thm:effective_flow_rate}.

This paper is organized as follows. Sect.~\ref{sec:second_order_traffic_models} consists of a brief presentation of the two types of second order traffic models studied here, as well as their mathematical properties that are relevant for the construction of jamitons. In Sect.~\ref{sec:stability}, we outline the aforementioned key theorem in the study of second order systems of hyperbolic balance laws with a relaxation term: the equivalence of the instability of uniform base states, and the sub-characteristic condition. Then, for common choices of parameters in the considered second order traffic models, the implications of this relationship are outlined. The explicit construction of jamiton solutions in the two types of second order traffic models is presented in Sect.~\ref{sec:jamiton_construction}. In particular, the equivalence between the breakdown of the sub-characteristic condition and the existence of jamitons is shown. In Sect.~\ref{sec:jamitons_fundamental_diagram}, we demonstrate how the preceding results about jamitons and the stability of uniform traffic flow can be employed to systematically construct set-valued fundamental diagrams that qualitatively reproduce the key features of fundamental diagrams obtained from measurements. Finally, in Sect.~\ref{sec:conclusions_and_outlook}, conclusions are drawn and an outlook towards extensions of the presented work is given.

\vspace{1.5em}
\section{Second Order Traffic Models with Relaxation Terms}
\label{sec:second_order_traffic_models}
Macroscopic traffic models describe the temporal evolution of the vehicle density $\rho(x,t)$, where $x$ is the position along a road (multiple lanes are aggregated into a single variable), and $t$ is time. Letting $u(x,t)$ denote the mean flow velocity of the moving vehicles, the density function evolves according to the continuity equation
\begin{equation}
\label{eq:continuity_equation}
\rho_t+(\rho u)_x = 0\;.
\end{equation}
First order traffic models, such as the Lighthill-Whitham-Richards (LWR) model \cite{LighthillWhitham1955, Richards1956}, assume a functional relationship between $\rho$ and $u$, i.e.~$u = U(\rho)$. We call $U$ the \emph{desired velocity function}, and assume it to be decreasing with $\rho$. A simple choice is
\begin{equation}
\label{eq:desired_velocity_linear}
U(\rho) = u_\text{max}(1-\rho/\rho_\text{max})\;,
\end{equation}
which is a linear interpolation between a maximum velocity $u_\text{max}$ of a single vehicle on an otherwise empty road (i.e., $\rho = 0$) and a stagnated flow at some maximum density, i.e., $U(\rho_\text{max}) = 0$. Because \eqref{eq:desired_velocity_linear} does not match very well with measurement data (see Fig.~\ref{fig:fd_data}), a variety of other desired velocity functions has been proposed \cite{Greenberg1959, Underwood1961, Newell1993, Daganzo1994, WangPapageorgiou2005}. The choice of a desired velocity function $U(\rho)$ leads to a scalar hyperbolic conservation law model
\begin{equation}
\label{eq:lighthill_whitham_richards_model}
\rho_t+(Q(\rho))_x = 0\;,
\end{equation}
where the flux function $Q(\rho) = \rho U(\rho)$ describes a unique $Q$ vs.~$\rho$ relation in the associated fundamental diagram of traffic flow. In contrast, second order models augment \eqref{eq:continuity_equation} by an evolution equation for the velocity field $u$. Below, we outline two classical types of second order models: the Payne-Whitham model and the inhomogeneous Aw-Rascle-Zhang model.
\begin{rem}
These two different types of second order models are studied in order to demonstrate that the results of this paper are not limited to only one specific second order model, but instead hold at a wider level of generality. That being said, it is known that the PW model \eqref{eq:payne_whitham_model} can develop solutions with unrealistic features, such as negative velocities or shocks that overtake vehicles from behind \cite{Daganzo1995}. It is therefore an invalid model for velocities close to zero, or when used with initial conditions that create spurious shocks. However, the analysis in Sect.~\ref{subsec:jamiton_construction_payne_whitham_model} implies that solutions with unrealistic shocks cannot arise dynamically as parts of traveling wave solutions, see Rem.~\ref{rem:pw_model_shocks}. Results of this nature are of interest in particular for the understanding of simulators based on the PW model (c.f.~\cite{MessmerPapageorgiou1990}).
\end{rem}

\subsection{Payne-Whitham Model}
\label{subsec:second_order_traffic_models_payne_whitham_model}
The Payne-Whitham (PW) model \cite{Payne1971, Whitham1974}, written in the primitive variables $\rho$ and $u$, reads as
\begin{equation}
\label{eq:payne_whitham_model}
\begin{split}
\rho_t+(\rho u)_x &= 0\;, \\
u_t+uu_x+\tfrac{1}{\rho}p(\rho)_x &= \tfrac{1}{\tau}\prn{U(\rho)-u}\;,
\end{split}
\end{equation}
where the \emph{traffic pressure} $p(\rho)$ is a strictly increasing function and $\tau$ is a relaxation time that dictates how rapidly drivers adjust to their desired velocity $U(\rho)$. The PW model \eqref{eq:payne_whitham_model} is a system of hyperbolic balance laws that possesses a relaxation term in the velocity equation. In analogy to fluid dynamics, the PW model is generally taken to possess the conservative quantities $\rho$ and $q = \rho u$, in which it reads as
\begin{equation}
\label{eq:payne_whitham_model_conservative}
\begin{split}
\rho_t+q_x &= 0\;, \\
q_t+\prn{\tfrac{q^2}{\rho}+p(\rho)}_x &= \tfrac{1}{\tau}\prn{\rho U(\rho)-q}\;.
\end{split}
\end{equation}
While \eqref{eq:payne_whitham_model} and \eqref{eq:payne_whitham_model_conservative} are equivalent for smooth solutions, the propagation of shocks is determined by the Rankine-Hugoniot conditions \cite{Evans1998} associated with the form \eqref{eq:payne_whitham_model_conservative}. The characteristic velocities of the PW model are
\begin{equation}
\label{eq:payne_whitham_model_characteristic_velocities}
\lambda_1 = \tfrac{q}{\rho}-\sqrt{p'(\rho)}
\quad\text{and}\quad
\lambda_2 = \tfrac{q}{\rho}+\sqrt{p'(\rho)}\;,
\end{equation}
i.e., in the PW model, information travels with the velocities $u\pm c(\rho)$, where $c(\rho) = \sqrt{p'(\rho)}$.

System \eqref{eq:payne_whitham_model_conservative} is the PW model written in Eulerian variables $\rho(x,t)$ and $q(x,t)$. An alternative description, as for instance employed in \cite{Greenberg2004}, is to use Lagrangian variables $\vv(\sigma,t)$ and $u(\sigma,t)$. Here $\sigma$ is the vehicle number, defined so that
\begin{equation}
\label{eq:relation_sigma_x_t}
\ud{\sigma} = \rho\ud{x}-\rho u\ud{t}\;,
\end{equation}
and $\vv = 1/\rho$ is the specific traffic volume, i.e., the road length per vehicle. In these variables, the PW model reads as
\begin{equation}
\label{eq:payne_whitham_model_lagrangian}
\begin{split}
\vv_t - u_\sigma &= 0\;, \\
u_t + (\hat{p}(\vv))_\sigma &= \tfrac{1}{\tau}\prn{\hat{U}(\vv)-u}\;,
\end{split}
\end{equation}
where $\hat{p}(\vv) = p(1/\vv)$ and $\hat{U}(\vv) = U(1/\vv)$. For notational economy, we shall hereafter omit the hat in these two functions, except when required for clarity.

The Rankine-Hugoniot jump conditions associated with \eqref{eq:payne_whitham_model_lagrangian} are
\begin{equation}
\label{eq:payne_whitham_model_rankine_hugoniot_conditions}
\begin{split}
m\brk{\vv}-\brk{u} &= 0\;, \\
m\brk{u}+\brk{p(\vv)} &= 0\;.
\end{split}
\end{equation}
Here $\brk{z}$ is the magnitude of the jump in the variable $z$ across a shock, and $-m$ is the propagation speed of the shock (in the Lagrangian, i.e., vehicle-centered, frame of reference).

\subsection{Inhomogeneous Aw-Rascle-Zhang Model}
\label{subsec:second_order_traffic_models_aw_rascle_zhang_model}
The inhomogeneous Aw-Rascle-Zhang model \cite{AwRascle2000, Zhang2002, Greenberg2001} in primitive Eulerian variables reads as
\begin{equation}
\label{eq:aw_rascle_zhang_model}
\begin{split}
\rho_t+(\rho u)_x &= 0\;, \\
(u+h(\rho))_t+u(u+h(\rho))_x &= \tfrac{1}{\tau}\prn{U(\rho)-u}\;,
\end{split}
\end{equation}
where the \emph{hesitation function} $h(\rho)$ is a strictly increasing function. As in the PW model, one assumes, in a loose analogy to fluid dynamics, the conservative quantities to be $\rho$ and $q = \rho (u+h(\rho))$. In these variables, the model reads as
\begin{equation}
\label{eq:aw_rascle_zhang_model_conservative}
\begin{split}
\rho_t+q_x &= 0\;, \\
q_t+\prn{\tfrac{q^2}{\rho}-h(\rho)q}_x &= \tfrac{1}{\tau}\prn{\rho (U(\rho)+h(\rho))-q}\;.
\end{split}
\end{equation}
The characteristic velocities of the ARZ model are
\begin{equation}
\label{eq:aw_rascle_zhang_model_characteristic_velocities}
\lambda_1 = \tfrac{q}{\rho}-h(\rho)-\rho h'(\rho) = u-\rho h'(\rho)
\quad\text{and}\quad
\lambda_2 = \tfrac{q}{\rho}-h(\rho) = u\;,
\end{equation}
where $u = \tfrac{q}{\rho}-h(\rho)$. Thus, in the ARZ model, information propagates along one family of characteristics with the vehicles, and along the second family of characteristics against traffic (relative to the vehicles) with velocity $\rho h'(\rho)$.

In Lagrangian variables, the ARZ model becomes
\begin{equation}
\label{eq:aw_rascle_zhang_model_lagrangian}
\begin{split}
\vv_t - u_\sigma &= 0\;, \\
(u+\hat{h}(\vv))_t &= \tfrac{1}{\tau}\prn{\hat{U}(\vv)-u}\;,
\end{split}
\end{equation}
where $\hat{h}(\vv) = h(1/\vv)$ and $\hat{U}(\vv) = U(1/\vv)$. Again, we shall generally omit these hats from now on. The associated Rankine-Hugoniot conditions are
\begin{equation}
\label{eq:aw_rascle_zhang_model_rankine_hugoniot_conditions}
\begin{split}
m\brk{\vv}-\brk{u} &= 0\;, \\
\brk{u}+\brk{h(\vv)} &= 0\;.
\end{split}
\end{equation}

\subsection{Assumptions on the Desired Velocity, Pressure, and Hesitation}
\label{subsec:assumptions_U_p_h}
In addition to requiring that $U(\rho)$, $p(\rho)$, and $h(\rho)$ be twice continuously differentiable, we impose the following assumptions:
\begin{enumerate}[ (a)]
\item
$\frac{\ud{U}}{\ud{\vv}}>0$ and $\frac{\ud{}^2U}{\ud{\vv^2}}<0$. This is equivalent to $U(\rho)$ being decreasing, and the flux function $Q(\rho) = \rho U(\rho)$ in the corresponding LWR model \eqref{eq:lighthill_whitham_richards_model} being concave. The concavity of $Q(\rho)$ is a frequent assumption in traffic flow (see for instance the discussions in \cite{LighthillWhitham1955} and \cite{Whitham1974}). In particular, it guarantees a nice shock theory for the LWR model, where the shocks satisfy the Lax entropy condition \cite{LeVeque1992} strictly.
\item
$\frac{\ud{p}}{\ud{\vv}}<0$ and $\frac{\ud{}^2p}{\ud{\vv^2}}>0$. These conditions lead to well-posed equations for the PW model \eqref{eq:payne_whitham_model_lagrangian}, even in the presence of shocks \cite{LeVeque1992}. These conditions are equivalent to: as functions of $\rho$, the pressure $p(\rho)$ is increasing and the function $\rho\,p(\rho)$ is convex. \item
$\frac{\ud{h}}{\ud{\vv}}<0$ and $\frac{\ud{}^2h}{\ud{\vv^2}}>0$, for the same reasons as for $p$.
\end{enumerate}
Note that all the example functions $U$, $p$, and $h$ used in this paper satisfy these assumptions.

\vspace{1.5em}
\section{Instability of Uniform Flow and Sub-Characteristic Condition}
\label{sec:stability}
A particular feature of balance laws of the form \eqref{eq:payne_whitham_model} and \eqref{eq:aw_rascle_zhang_model} is the relaxation term in the (generalized) momentum equation, which is expected to drive the actual velocity $u$ towards the desired velocity $U(\rho)$ when $\tau\ll 1$. If this is indeed the case, the continuity equation \eqref{eq:continuity_equation} shows that the solution will be very close to the one of the LWR model \eqref{eq:lighthill_whitham_richards_model} with the flux function $Q(\rho) = \rho U(\rho)$.

Note that \eqref{eq:lighthill_whitham_richards_model} has a single characteristic velocity
\begin{equation}
\label{eq:lighthill_whitham_richards_model_characteristic_velocity}
\mu = Q'(\rho) = U(\rho)+\rho U'(\rho)\;,
\end{equation}
while the second order models \eqref{eq:payne_whitham_model} or \eqref{eq:aw_rascle_zhang_model} possess two characteristic velocities $\lambda_1<\lambda_2$, which depend in particular on the traffic pressure $p(\rho)$, or the hesitation function $h(\rho)$, respectively.

Moreover, one can observe that all three models \eqref{eq:payne_whitham_model}, \eqref{eq:aw_rascle_zhang_model}, and \eqref{eq:lighthill_whitham_richards_model} share a set of solutions, namely the \emph{base state solutions}, in which $\rho(x,t) = \bar{\rho} = \text{const.}$ and $u = U(\bar{\rho}) = \text{const}$. These solutions represent uniform traffic flow with equi-spaced vehicles (whose midpoints are a distance $1/\bar{\rho}$ apart). Let now $\rho = \bar{\rho}$, $u = \bar{u} = U(\bar{\rho})$ be any given base state solution of a second order model and consider the following two conditions that can apply:
\begin{itemize}
\item[\textbf{LS}] (linear stability) \\
The base state $(\bar{\rho},\bar{u})$ is linearly stable, i.e., any infinitesimal perturbations to it decay in time.
\item[\textbf{SCC}] (sub-characteristic condition) \\
The characteristic velocity $\mu = \mu(\bar{\rho})$, given by \eqref{eq:lighthill_whitham_richards_model_characteristic_velocity}, of the \emph{reduced} LWR model \eqref{eq:lighthill_whitham_richards_model} falls in between the two characteristic velocities $\lambda_j = \lambda_j(\bar{\rho},\bar{u})$ of the second order model, i.e.
\begin{equation}
\label{eq:sub_characteristic_condition}
\lambda_1 < \mu < \lambda_2\;.
\end{equation}
\end{itemize}
These two conditions are related by the following
\begin{thm}[Whitham]
\label{thm:whitham}
The two conditions LS and SCC are equivalent.
\end{thm}
The proof, given by Whitham in 1959 \cite{Whitham1959}, and thoroughly reviewed in his 1974 book \cite[Chap.~10]{Whitham1974}, is constructive: for a general $2\times 2$ hyperbolic conservation law with a relaxation term in one equation, the standard linear stability analysis (calculating the growth factor for basic wave perturbations of the form $e^{ikx}$) is performed, and subsequently related to the sub-characteristic condition.

Besides the theoretical insight that Thm.~\ref{thm:whitham} yields (see the discussion in \cite{Whitham1974}), it allows one to substitute the task of performing a linear stability analysis by, for instance, normal modes (e.g., as done for \eqref{eq:payne_whitham_model} in \cite[App.~B]{FlynnKasimovNaveRosalesSeibold2009} and for \eqref{eq:aw_rascle_zhang_model} in \cite{Greenberg2004}) by checking the sub-characteristic condition, which is often times simpler to verify.

In the context of traffic flow, an important question is: when can one use a first order model and (essentially) get the same answer as that given by a second order model? This question, which turns out to be intimately related to the result in Thm.~\ref{thm:whitham}, has been investigated by several authors, e.g., Liu \cite{Liu1987}, Chen, Levermore \& Liu \cite{ChenLevermoreLiu1994}, and Li \& Liu \cite{LiLiu2005} (note that these papers consider situations that are more general than the ones arising for traffic flow). Thus, for example, it is known that:
\begin{enumerate}[ 1.]
\item
Solutions of a second order traffic model with uniform base state initial conditions, plus a sufficiently small perturbation, converge towards this base state (i.e., the perturbation decays) as $t\to\infty$, if the base state satisfies the SCC. In other words, the equivalence in Thm.~\ref{thm:whitham} can be extended from relating SCC and LS, to relating SCC and stability.
\item
The solutions to second order models that satisfy the SCC everywhere, converge towards solutions of the reduced LWR model in the limit $\tau\to 0$. More generally, if $\tau$ is not small, then the solutions converge towards the solutions of a reduced LWR model with a nonlinear ``viscosity'', which is of magnitude $O(\tau)$.
\end{enumerate}
An important question that the work cited above does not address is:
\begin{enumerate}[ 3.]
\item
If constant base state (plus small perturbation) initial conditions (strictly) violate the SCC, how do the solutions to a second order model evolve in time?
\end{enumerate}
To the authors' knowledge, no rigorous results are known in this case, except for the fact that the solutions to the second order models need not look like solutions to the reduced LWR model \eqref{eq:lighthill_whitham_richards_model}, even if $\tau\ll 1$. However:
\begin{enumerate}[ 4.]
\item
Numerical studies (for the PW model in \cite{FlynnKasimovNaveRosalesSeibold2009}, and for the ARZ model in \cite{Greenberg2004}) indicate that the solutions of second order traffic flow models tend towards a jamiton-dominated regime, characterized by the presence of self-sustained finite amplitude traveling waves.
\end{enumerate}
The construction of these traveling waves is addressed in Sect.~\ref{sec:jamiton_construction}. However, before doing so, we calculate the sub-characteristic condition for the PW and the ARZ model below.
\begin{rem}
\label{rem:degenerate_SCC}
The analysis of the degenerate case of the SCC condition, i.e., $\lambda_1 = \mu < \lambda_2$ or $\lambda_1 < \mu = \lambda_2$, in \eqref{eq:sub_characteristic_condition}, is more involved (c.f.~\cite{Li2000, LiLiu2009}) and we do not consider it in this paper.
\end{rem}

\subsection{Payne-Whitham Model}
\label{subsec:stability_payne_whitham_model}
Using the characteristic velocities of the PW model, \eqref{eq:payne_whitham_model_characteristic_velocities}, and of the LWR model, \eqref{eq:lighthill_whitham_richards_model_characteristic_velocity}, the sub-characteristic condition \eqref{eq:sub_characteristic_condition} reduces to
\begin{equation*}
-\sqrt{p'(\rho)} < \rho U'(\rho) < \sqrt{p'(\rho)}\;.
\end{equation*}
Because $U(\rho)$ is decreasing and $p(\rho)$ is increasing, the second inequality is always satisfied, rendering the SCC, and thus the stability of uniform flow of density $\rho$, equivalent to
\begin{equation}
\label{eq:scc_payne_whitham}
\frac{p'(\rho)}{\rho^2} > (U'(\rho))^2\;.
\end{equation}
Regarding reasonable choices for $U(\rho)$ and $p(\rho)$, the desired velocity function has a clear physical interpretation. In this paper, we consider two forms of $U(\rho)$: the simple linear relationship \eqref{eq:desired_velocity_linear}, and a more complicated shape, given via \eqref{eq:smoothed_Newell_Daganzo_flux}. In contrast, which choices are reasonable for the traffic pressure is more difficult to motivate. For the linear velocity $U(\rho) = u_\text{max}(1-\rho/\rho_\text{max})$, we briefly study two classes of traffic pressures in terms of their behavior with respect to the (in)stability of base states. We argue that reasonable traffic models for phantom traffic jams must have the property that low density base states are stable (\emph{free flow}), instability occurs above a certain critical threshold density (\emph{congested flow} or \emph{synchronized flow}), and possibly stability re-occurs for densities close to the maximum density $\rho_\text{max}$ (\emph{wide moving jams}), where the terminology is in line with \cite{Kerner1998, Varaiya2005}.
We distinguish two types of pressures:
\begin{enumerate}[(A)]
\item\textbf{regular pressure}:
Drawing an analogy to gas dynamics, a simple choice is $p(\rho) = \beta\rho^\gamma$, where $\gamma>0$. In this case, the SCC becomes $\rho^{\gamma-3} > \frac{(u_\text{max})^2}{\beta\gamma(\rho_\text{max})^2}$. One can see that for $\gamma>3$, low densities violate the SCC, and for $\gamma=3$, no transitions between stability or instability occur. Hence, reasonable behavior occurs only for $\gamma<3$, for which instability occurs above the critical density $\rho_\text{crit} = (\beta\gamma(\rho_\text{max})^2/(u_\text{max})^2)^\frac{1}{3-\gamma}$. Some possible choices are $\gamma=1$, as in \cite{KernerKlenovKonhauser1997}, or $\gamma=2$, which renders the PW model \eqref{eq:payne_whitham_model} equivalent to the shallow water equations with a relaxation term.
\item\textbf{singular pressure:}
Stability for densities close to $\rho_\text{max}$ can be achieved by selecting a pressure $p(\rho)$ with a pole at $\rho_\text{max}$. For example, as in \cite{FlynnKasimovNaveRosalesSeibold2009}, we can take
\begin{equation}
\label{eq:logarithmic_pressure}
p(\rho) = -\beta(\rho/\rho_\text{max}+\log(1-\rho/\rho_\text{max}))\;,
\end{equation}
for which the SCC becomes $\frac{\beta/\rho_\text{max}}{\rho(\rho_\text{max}-\rho)} > \frac{(u_\text{max})^2}{(\rho_\text{max})^2}$, or equivalently $\frac{\rho}{\rho_\text{max}}(1-\frac{\rho}{\rho_\text{max}}) < \frac{\beta}{\rho_\text{max}(u_\text{max})^2}$. Thus, there is a range of unstable base states that is centered around $\tfrac{1}{2}\rho_\text{max}$, and is of width $((\tfrac{1}{2}\rho_\text{max})^2-\frac{\beta\rho_\text{max}}{(u_\text{max})^2})^\frac{1}{2}$. Hence, for low densities, and high densities (i.e, close to $\rho_\text{max}$), uniform traffic flow is stable. Thus, singular pressures can generate traffic regimes qualitatively close to those presented by Kerner \cite{Kerner1998}. We shall use this pressure \eqref{eq:logarithmic_pressure} for some of the examples presented in Sect.~\ref{sec:jamitons_fundamental_diagram}.
\end{enumerate}

\subsection{Inhomogeneous Aw-Rascle-Zhang Model}
\label{subsec:stability_aw_rascle_zhang_model}
Using the characteristic velocities of the ARZ model, \eqref{eq:aw_rascle_zhang_model_characteristic_velocities}, and of the LWR model, \eqref{eq:lighthill_whitham_richards_model_characteristic_velocity}, the sub-characteristic condition \eqref{eq:sub_characteristic_condition} reduces to
\begin{equation*}
-\rho h'(\rho) < \rho U'(\rho) < 0\;.
\end{equation*}
Because $U(\rho)$ is decreasing, the second inequality is always satisfied, rendering the SCC, and thus the stability of uniform flow of density $\rho$, equivalent to
\begin{equation}
\label{eq:scc_aw_rascle_zhang}
h'(\rho) > -U'(\rho)\;.
\end{equation}
As in the PW model, we study two types of hesitation functions, in conjunction with the desired velocity function $U(\rho) = u_\text{max}(1-\rho/\rho_\text{max})$:
\begin{enumerate}[(A)]
\item\textbf{regular hesitation}:
For the choice $h(\rho) = \beta\rho^\gamma$ with $\gamma>0$, the SCC becomes $\rho^{\gamma-1} > \frac{u_\text{max}}{\beta\gamma\rho_\text{max}}$. Similarly as with the PW model, realistic behavior of (in)stability regimes only occurs if $\gamma<1$.
\item\textbf{singular hesitation:}
As with the PW model, stability can be recovered close to $\rho_\text{max}$ by letting $h(\rho)$ become singular there. A possible choice, suggested in \cite{BerthelinDegondDelitalaRascle2008}, is
\begin{equation}
\label{eq:BDDR_hesitation}
h(\rho) = \beta\prn{\tfrac{\rho/\rho_\text{max}}{1-\rho/\rho_\text{max}}}^\gamma\;.
\end{equation}
Since for $\rho\ll 1$, the function $h(\rho)$ behaves as in the regular case, one can again only achieve realistic (in)stability regimes if $\gamma<1$. In this case, $h(\rho)$ possesses an inflection point and thus can generate the traffic regimes presented by Kerner \cite{Kerner1998}. In Sect.~\ref{sec:jamitons_fundamental_diagram}, we shall employ \eqref{eq:BDDR_hesitation} with $\gamma=\frac{1}{2}$.
\end{enumerate}

\vspace{1.5em}
\section{Construction of Jamiton Solutions}
\label{sec:jamiton_construction}
Jamitons are self-sustained traveling waves in second order traffic flow models with a relaxation term. They possess the same mathematical structure as detonation waves in the Zel'dovich-von~Neumann-D{\"o}ring (ZND) theory \cite{FickettDavis1979}. A ZND detonation wave consists of a compressive shock with a smooth reaction zone attached to it, which is traveling at the same velocity as the shock. The ideas presented in this section are an adaptation of this theory to the PW and the ARZ traffic models. A theory valid for generic second order traffic flow models is presented in a companion paper \cite{KasimovRosalesSeiboldFlynn2012}.

In order to obtain expressions for jamitons, a traveling wave ansatz is performed. This approach has been presented in \cite{FlynnKasimovNaveRosalesSeibold2009} in Eulerian variables, wherein one considers solutions of the form $\rho(x,t) = \hat{\rho}(\eta)$ and $u(x,t) = \hat{u}(\eta)$, where $\eta = \frac{x-st}{\tau}$ is the self-similar variable. Here, we use the Lagrangian variables defined in Sect.~\ref{sec:second_order_traffic_models}, which have in particular been employed in Greenberg's study of the ARZ model \cite{Greenberg2004}. Note that, while the analyses in Lagrangian and in Eulerian variables are mathematically equivalent, it is the former choice of variables that affords the clear geometric interpretations and constructions conducted below.

\subsection{Payne-Whitham Model}
\label{subsec:jamiton_construction_payne_whitham_model}
We consider system \eqref{eq:payne_whitham_model_lagrangian} and seek solutions of the form $\vv(x,t) = \hat{\vv}(\chi)$ and $u(x,t) = \hat{u}(\chi)$, where $\chi = \frac{mt+\sigma}{\tau}$ is the self-similar variable. For simplicity of notation, we shall omit the hats in the following. Substituting this self-similar form into \eqref{eq:payne_whitham_model_lagrangian} yields the relations
\begin{align}
\label{eq:payne_whitham_model_lagrangian_self_similar_1}
\tfrac{m}{\tau}\vv'(\chi) - \tfrac{1}{\tau}u'(\chi) &= 0\;, \\
\label{eq:payne_whitham_model_lagrangian_self_similar_2}
\tfrac{m}{\tau}u'(\chi) + p'(\vv(\chi))\tfrac{1}{\tau}\vv'(\chi) &= \tfrac{1}{\tau}\prn{U(\vv(\chi))-u(\chi)}\;.
\end{align}
Relation \eqref{eq:payne_whitham_model_lagrangian_self_similar_1} implies that
\begin{equation}
\label{eq:relation_v_u}
m\vv-u = -s\;,
\end{equation}
where $s$ is a constant of integration.
\begin{rem}
In Lagrangian variables, $-m$ is the propagation velocity of jamitons, i.e., jamitons travel backward at speed $m$ with respect to the moving vehicles. The constant $s$ has the interpretation of a ``Lagrangian mass flux''. In Eulerian variables, the meaning of these two constants is exactly reversed: jamitons move at velocity $s$ with respect to a fixed observer, and $m$ is the ``mass'' flux of vehicles past a stationary observer. Note that due to the entropy condition described in Sect.~\ref{subsubsec:shocks}, solutions of the type above can exist only for $m>0$. On the other hand, there is no sign restriction on $s$.
\end{rem}
Using relation \eqref{eq:relation_v_u} to substitute $u$ by $\vv$ in \eqref{eq:payne_whitham_model_lagrangian_self_similar_2}, we obtain a first order ODE for $\vv(\chi)$, which reads as
\begin{equation}
\label{eq:jamiton_ode}
\vv'(\chi) = \frac{w(\vv(\chi))}{r'(\vv(\chi))}\;,
\end{equation}
where the two functions $w$ and $r$ are defined as
\begin{equation}
\label{eq:definition_w_r}
w(\vv) = U(\vv)-(m\vv+s)
\quad\text{and}\quad
r(\vv) = p(\vv)+m^2\vv\;,
\end{equation}
and depicted in Fig.~\ref{fig:wv_diagram} and Fig.~\ref{fig:rv_diagram}, respectively. At first glance, it appears that the constants $m$ and $s$ can be chosen independently. However, this is not the case. To see this, we use the assumptions $p'(\vv)<0$ and $p''(\vv)>0$ (see Sect.~\ref{subsec:assumptions_U_p_h}). Then, the denominator in \eqref{eq:jamiton_ode}, $r'(\vv)$, may have at most one root, $\vS$, such that $p'(\vS) = -m^2$. The ODE \eqref{eq:jamiton_ode} can only be integrated through $\vS$, if the numerator in \eqref{eq:jamiton_ode}, $w(\vv)$, has a (simple) root at $\vS$ as well, i.e.
\begin{equation}
\label{eq:chapman_jouguet_condition}
m\vS+s = U(\vS)\;.
\end{equation}
In ZND detonation theory, relation \eqref{eq:chapman_jouguet_condition} is known as the \emph{Chapman-Jouguet condition} \cite{FickettDavis1979}, and the point at which $r'(\vv)$ vanishes is called \emph{sonic point}. We can therefore parameterize the smooth traveling wave solutions of the jamiton ODE \eqref{eq:jamiton_ode} by the \emph{sonic value} $\vS$. Given $\vS$, the two constants are
\begin{equation}
\label{eq:values_m_s_payne_whitham_model}
m = \sqrt{-p'(\vS)} \quad\text{and}\quad s = U(\vS)-m\vS\;.
\end{equation}
Hence, we obtain a one-parameter family of smooth traveling wave solutions, parameterized by $\vS$, each of which is a solution of \eqref{eq:jamiton_ode}.

Whether, in the vicinity of $\vS$, the function $\vv(\chi)$ is increasing or decreasing, is determined by the sign of the right hand side in \eqref{eq:jamiton_ode}, which, by L'H{\^o}pital's rule, equals
\begin{equation*}
\frac{w'(\vS)}{r''(\vS)} = \frac{U'(\vS)-m}{p''(\vS)}\;.
\end{equation*}
Because $p''(\vv)>0$, we note that $\vv(\chi)$ is increasing if $U'(\vS)>m$, and decreasing if $U'(\vS)<m$; the degenerate case $U'(\vS) = m$ yields only constant solutions.

\begin{figure}
\begin{minipage}[b]{.485\textwidth}
\includegraphics[width=\textwidth]{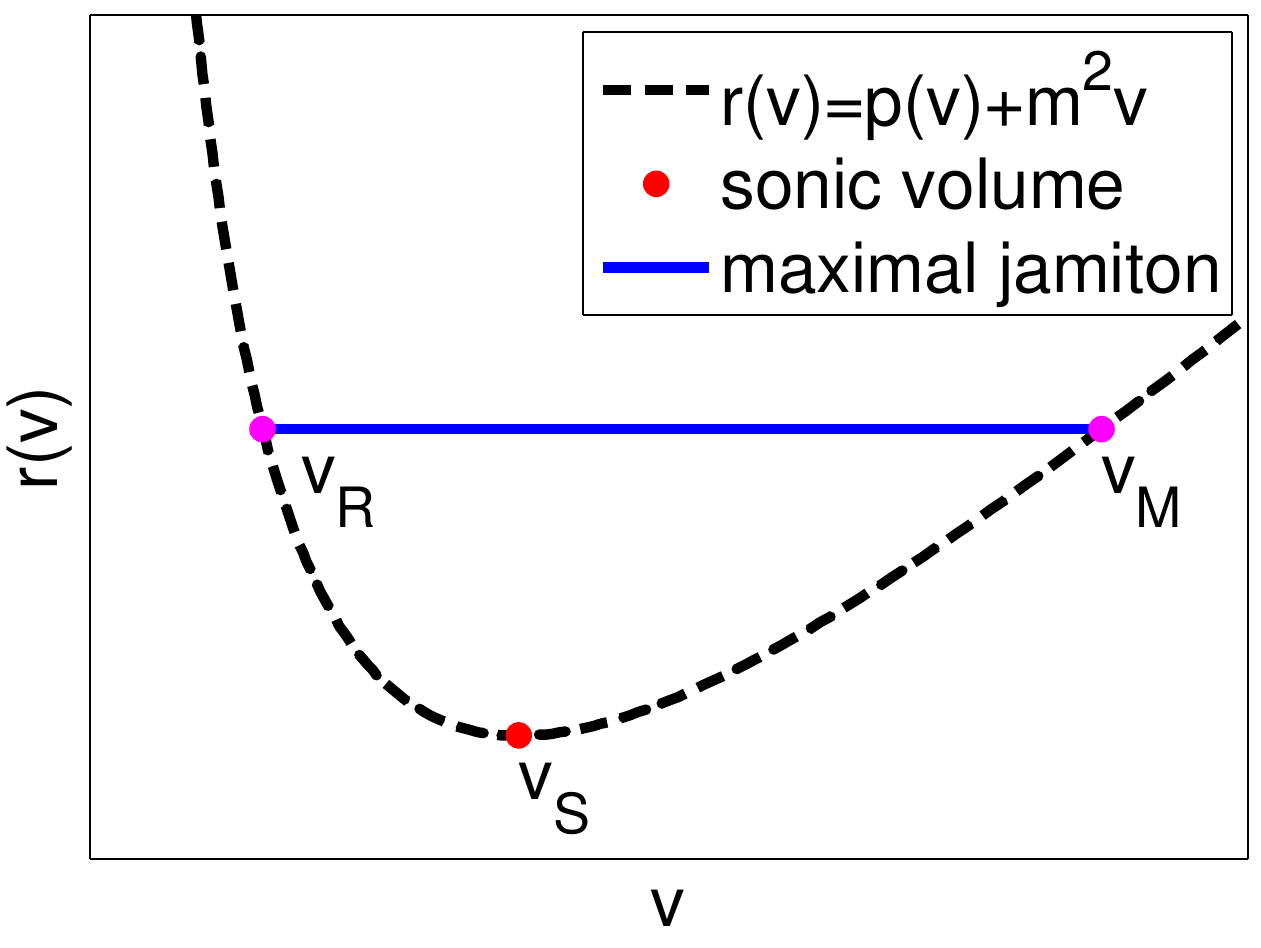}
\vspace{-1.8em}
\caption{The $r$--$\vv$ diagram in the $p$--$\vv$ plane. The sonic value $\vS$ is located in the minimum of $r(\vv)$, and the maximal jamiton (blue line) connects the states $\vv_\text{R}$ and $\vv_\text{M}$.}
\label{fig:rv_diagram}
\end{minipage}
\hfill
\begin{minipage}[b]{.485\textwidth}
\includegraphics[width=\textwidth]{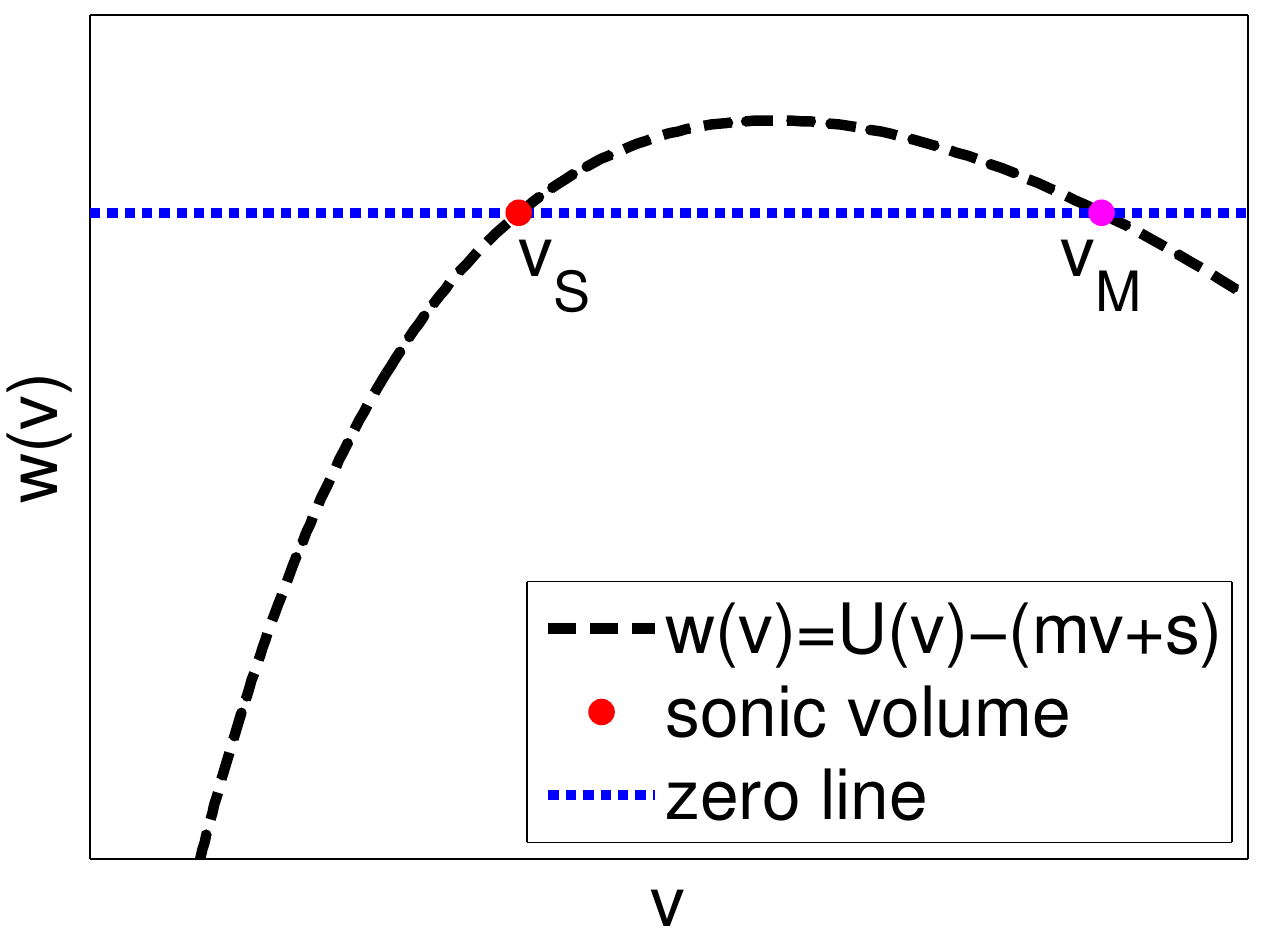}
\vspace{-1.8em}
\caption{The $w$--$\vv$ diagram in the $u$--$\vv$ plane. Jamitons exist exactly if $w(\vv)$ possesses a second root $\vv_\text{M}$ right of $\vS$, which is the case if $U'(\vS)>m$.}
\label{fig:wv_diagram}
\end{minipage}
\end{figure}

Using that $U''(\vv)<0$ (see Sect.~\ref{subsec:assumptions_U_p_h}), we know about the numerator in \eqref{eq:jamiton_ode}, $w(\vv)$, that: if $U'(\vS)>m$, a second root, $\vv_\text{M}$, can exist for $\vv>\vS$, but no further roots can exist for $\vv<\vS$ (see Fig.~\ref{fig:wv_diagram}). If, on the other hand, $U'(\vS)<m$, a second root, $\vv_\text{M}$, can exist for $\vv<\vS$, but no further roots can exist for $\vv>\vS$. As a consequence, the solution of \eqref{eq:jamiton_ode} can extend to infinity in the direction $\chi\to +\infty$ (for which $\vv\to\vv_\text{M}$). In turn, in the direction $\chi\to -\infty$ the solution approaches a vacuum state ($\vv\to\infty$), the jamming density ($\vv\to (\rho_\text{max})^{-1}$), or an infinite density ($\vv\to 0$), depending on the type of pressure used (see Sect.~\ref{subsec:stability_payne_whitham_model}). In this paper we consider none of these (rather extreme) situations, and instead focus on solutions that possess shocks, see below.
\begin{rem}
A complete classification of traveling wave solutions (with or without shocks) is provided in a companion paper \cite{SeiboldRosalesFlynnKasimov2013}.
\end{rem}

\subsubsection{Shocks}
\label{subsubsec:shocks}
Because we wish to consider only traveling waves that do not approach an extremal state for $\chi\to -\infty$, a shock must exist that connects a state $(\vv^-,u^-)$ to the left of the jump to a state $(\vv^+,u^+)$ to the right of the jump. A shock can be part of the traveling wave solution, if it travels with the same velocity $m$. The Rankine-Hugoniot jump conditions \eqref{eq:payne_whitham_model_rankine_hugoniot_conditions} now impose two restrictions on the states connected by the shock, as follows. The first condition in \eqref{eq:payne_whitham_model_rankine_hugoniot_conditions} implies that $m\vv^- - u^- = m\vv^+ - u^+$, i.e., the quantity $m\vv-u$ is conserved across the shock. This is in line with condition \eqref{eq:relation_v_u} for the smooth part of a traveling wave. A second requirement is obtained by combining the two jump conditions in \eqref{eq:payne_whitham_model_rankine_hugoniot_conditions} to yield
\begin{equation*}
m^2\brk{\vv} + \brk{p(\vv)} = 0\;,
\end{equation*}
which is equivalent to the statement that $r(\vv^-) = r(\vv^+)$, where $r$ is precisely the function defined in \eqref{eq:definition_w_r}. Therefore, any jamiton shock connects two states $\vv^-$ and $\vv^+$ by a horizontal line in the $r$--$\vv$ diagram, as shown in Fig.~\ref{fig:rv_diagram_jamiton}.

In addition to the Rankine-Hugoniot conditions, a stable shock must also satisfy the Lax entropy conditions (see e.g.~\cite{Whitham1974, LeVeque1992}). For a shock that is part of a traveling wave (see \cite{FlynnKasimovNaveRosalesSeibold2009} for a visualization), this requires that the specific volume decreases across a shock, i.e., $\vv^- > \vv^+$. Hence, in a smooth part of a traveling wave, $\vv$ must be increasing with $\chi$, and therefore a necessary condition for the existence of jamitons is that $U'(\vS)>m$ (see Fig.~\ref{fig:wv_diagram}). As a consequence, a jamiton solution $\vv(\chi)$ corresponds to a smooth movement along the function $r(\vv)$ in the $p$--$\vv$ plane into the positive $\vv$ direction, followed by a jump back from the right branch to the left branch along a horizontal line (see Fig.~\ref{fig:rv_diagram_jamiton}).
\begin{rem}
\label{rem:pw_model_shocks}
The PW model has been criticized \cite{Daganzo1995, AwRascle2000} for the fact that its solutions may involve shocks that travel faster than the vehicles. However, the arguments above point to the following interesting fact: while such unrealistic shocks can certainly be triggered by particular choices of initial conditions, they cannot occur as parts of traveling wave solutions (see also \cite{FlynnKasimovNaveRosalesSeibold2009}).
\end{rem}

\begin{figure}
\begin{minipage}[b]{.485\textwidth}
\includegraphics[width=\textwidth]{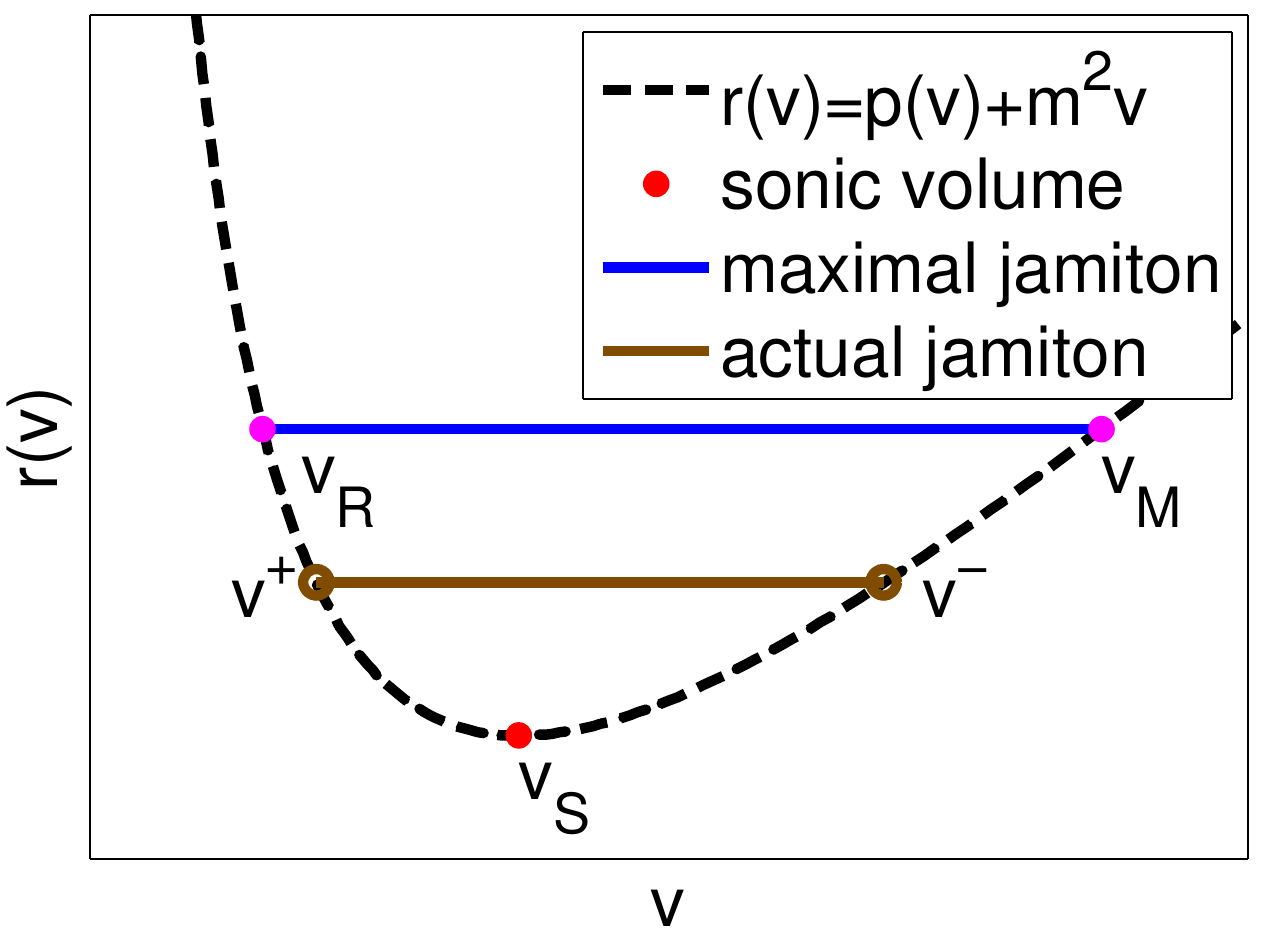}
\vspace{-1.8em}
\caption{A jamiton in the $r$--$\vv$ diagram in the $p$--$\vv$ plane. Below the maximal jamiton (blue line) lies an actual jamiton (brown line) that connects the states $\vv^+$ and $\vv^-$.}
\label{fig:rv_diagram_jamiton}
\end{minipage}
\hfill
\begin{minipage}[b]{.485\textwidth}
\includegraphics[width=\textwidth]{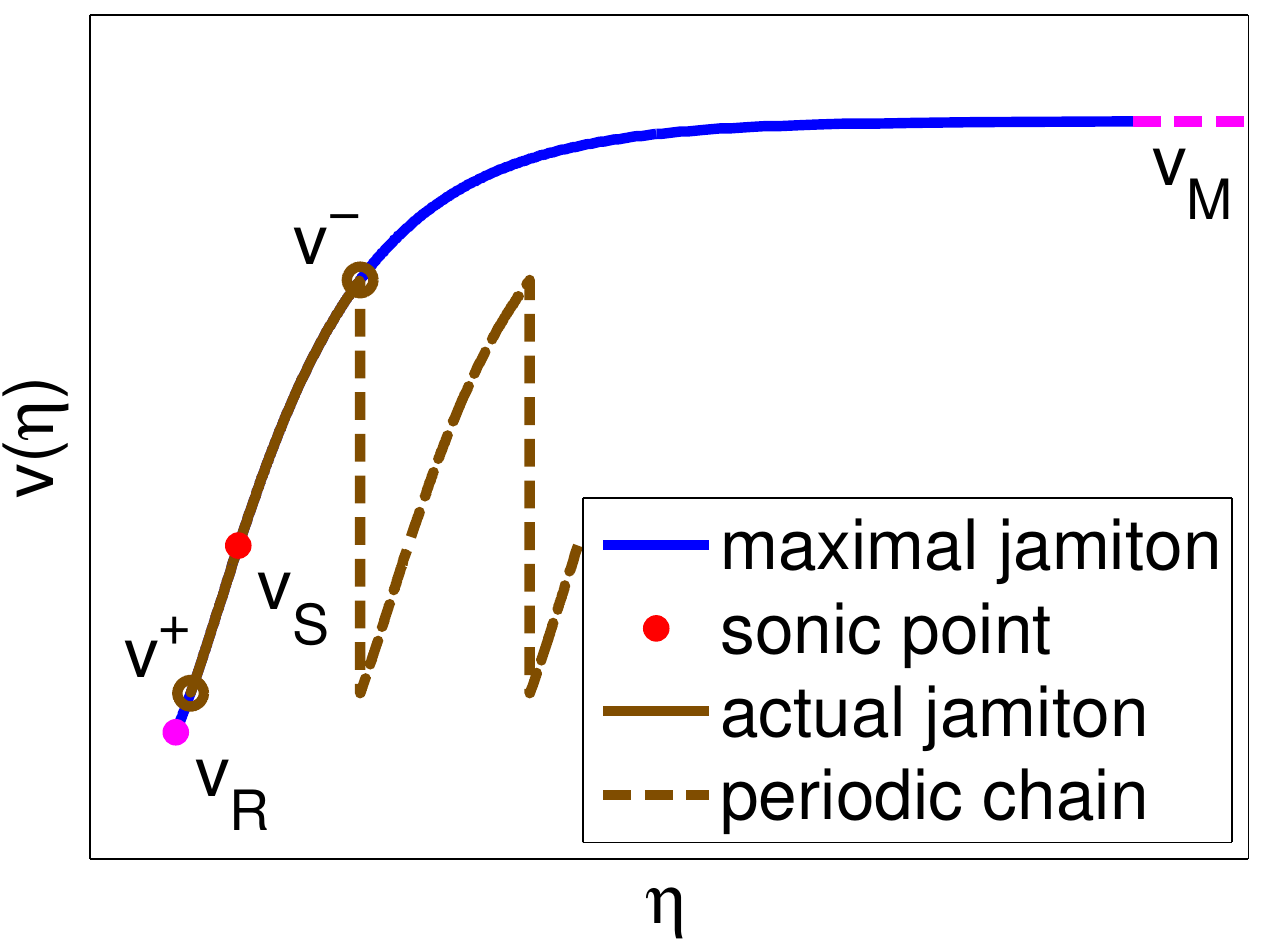}
\vspace{-1.8em}
\caption{Jamiton function $\vv(\eta)$. An actual jamiton (brown curve) between $\vv^+$ and $\vv^-$ lies on the maximal jamiton curve (blue), which connects $\vv_\text{R}$ and $\vv_\text{M}$, the latter being at $\eta\to\infty$.}
\label{fig:jamiton_curve}
\end{minipage}
\end{figure}

\subsubsection{Jamiton Construction}
\label{subsubsec:jamiton_construction}
Based on the theoretical presentation above, next we present a systematic process for the construction of jamiton solutions, with an intuitive graphical interpretation. The construction is as follows:
\begin{enumerate}[\quad 1.]
\item Select a sonic specific volume $\vS$.
\item Evaluate $m = \sqrt{-p'(\vS)}$ and $s = U(\vS)-m\vS$.
\item Draw the function $w(\vv) = U(\vv)-(m\vv+s)$ in the $u$--$\vv$ plane (see Fig.~\ref{fig:wv_diagram}).
\item If $w'(\vS)\le 0$, then no jamitons exist with this sonic value $\vS$. Otherwise, find the unique root $\vv_\text{M}>\vS$ such that $w(\vv_\text{M}) = 0$.
\item Draw the function $r(\vv) = p(\vv)+m^2\vv$ in the $p$--$\vv$ plane (see Fig.~\ref{fig:rv_diagram}), and define $r_\text{min} = r(\vS)$ and $r_\text{max} = r(\vv_\text{M})$.
\item Find the unique value $\vv_\text{R}<\vS$ that satisfies $r(\vv_\text{R}) = r_\text{max}$.
\newcounter{enumii_saved}
\setcounter{enumii_saved}{\value{enumi}}
\end{enumerate}
Since the jamiton ODE \eqref{eq:jamiton_ode} cannot be integrated beyond $\vv_\text{M}$, there is a \emph{maximal jamiton} (for this $\vS$) that connects the specific volumes $\vv_\text{M}$ and $\vv_\text{R}$, as depicted by the blue curve in Fig.~\ref{fig:jamiton_curve}. Note that this maximal jamiton is only of theoretical relevance because it is infinitely long. However, the above construction gives rise to a one-parameter family (for the specific $\vS$ at hand) of jamitons of finite length, whose smooth parts are comprised of the smooth part of the maximal jamiton, but which differ in the states that are connected by a shock. Specifically, the construction continues as follows:
\begin{enumerate}[\quad 1.]
\setcounter{enumi}{\value{enumii_saved}}
\item Select a value $\mathring{r}\in (r_\text{min},r_\text{max})$.
\item Determine $\vv^+<\vS$ and $\vv^->\vS$, such that $r(\vv^-) = \mathring{r} = r(\vv^+)$ (see Fig.~\ref{fig:rv_diagram_jamiton}).
\item Calculate desired properties of the jamiton, such as:
  \begin{itemize}
  \item Solve the jamiton ODE \eqref{eq:jamiton_ode} from $\vv^+$ to $\vv^-$ to plot the jamiton in the $\chi$ variable.
  \item Transform the jamiton into the Eulerian self-similar variable $\eta$, by using the relation
  \begin{equation*}
  \ud{\eta} = \tfrac{1}{\tau}(\ud{x}-s\ud{t})
  \overset{\eqref{eq:relation_v_u}}{=} \tfrac{1}{\tau}(\ud{x}-u\ud{t}+m\vv\ud{t})
  \overset{\eqref{eq:relation_sigma_x_t}}{=}
  \tfrac{1}{\tau}(\vv\ud{\sigma}+m\vv\ud{t})
  = \vv\ud{\chi}\;.
  \end{equation*}
  Figure~\ref{fig:jamiton_curve} shows the shape of a jamiton as a function $\vv(\eta)$, and Fig.~\ref{fig:jamiton_curve_rho_u} shows the corresponding jamiton profiles $\rho(\eta)$ and $u(\eta)$.
  \item Compute its length in the $\eta$ variable
  \begin{equation}
  \label{eq:jamiton_length}
  L = \eta(\vv^-)-\eta(\vv^+)
  = \int_{\vv^+}^{\vv^-} \eta'(\vv) \ud{\vv}
  = \int_{\vv^+}^{\vv^-} \vv\frac{r'(\vv)}{w(\vv)}\ud{\vv}
  \end{equation}
  and the number of vehicles in the jamiton (in the $\eta$ variable)
  \begin{equation}
  \label{eq:jamiton_mass}
  N = \int_{\eta^+}^{\eta^-} \rho(\eta) \ud{\eta}
  = \int_{\vv^+}^{\vv^-} \frac{\eta'(\vv)}{\vv} \ud{\vv}
  = \int_{\vv^+}^{\vv^-} \frac{r'(\vv)}{w(\vv)}\ud{\vv}
  \end{equation}
  by evaluating a definite integral. Of course, in line with having a complete continuum description, $N$ is not restricted to integer values. Moreover, in the physical road coordinate $x$, the jamiton length is $\tau L$, and the number of vehicles is $\tau N$.
  \end{itemize}
\end{enumerate}
As depicted through the dashed function in Fig.~\ref{fig:jamiton_curve}, a single jamiton can be continued into an infinite periodic chain of jamitons, all of which are connected through shocks. In the $r$--$\vv$ diagram shown in Fig.~\ref{fig:rv_diagram_jamiton}, this means that one moves continuously along the function $r(\vv)$ from $\vv^+$ to $\vv^-$, and then jumps back to $\vv^+$ via a shock.

In ZND detonation theory, the constant line whose intersection with the graph of $r(\vv)$ determines the two states $\vv^+$ and $\vv^-$, is called \emph{Rayleigh line}. As one can see in Fig.~\ref{fig:rv_diagram}, the function $r(\vv)$ is convex, and the sonic value $\vS$ is the unique point where its slope matches the slope of the Rayleigh line. The construction presented here leads us to the following
\begin{thm}
\label{thm:equivalence_jamitons_instability_payne_whitham_model}
For the PW model, the existence of jamitons with a sonic value $\vS$ is equivalent (with the exception of the degenerate case, see Rem.~\ref{rem:degenerate_SCC}) to the instability of the base state solution $\vv(x,t) = \vS$, $u(x,t) = U(\vS)$.
\end{thm}
\begin{proof}
As argued in Sect.~\ref{subsubsec:shocks}, a necessary condition for the existence of jamitons is that $U'(\vS)>m$, where $m = \sqrt{-p'(\vS)}$. This condition is the complement of the stability condition \eqref{eq:scc_payne_whitham}. Moreover, the construction in Sect.~\ref{subsubsec:jamiton_construction} shows that for every linearly unstable value $\vS$, one can in fact construct jamitons that have this sonic value $\vS$.
\end{proof}

\subsection{Inhomogeneous Aw-Rascle-Zhang Model}
\label{subsec:jamiton_construction_aw_rascle_zhang_model}
The construction of jamitons for the ARZ model is very similar to the construction for the PW model. Again, we seek solutions in the self-similar variable $\chi = \frac{mt+\sigma}{\tau}$. In Lagrangian variables, the ARZ model \eqref{eq:aw_rascle_zhang_model_lagrangian} differs from the PW model \eqref{eq:payne_whitham_model_lagrangian} only in the fact that $(p(\vv))_\sigma$ is replaced by $(h(\vv))_t$. Therefore, equation \eqref{eq:payne_whitham_model_lagrangian_self_similar_1} carries over, and thus relation \eqref{eq:relation_v_u} holds as well. Equation \eqref{eq:payne_whitham_model_lagrangian_self_similar_2} is replaced by
\begin{equation}
\label{eq:aw_rascle_zhang_model_lagrangian_self_similar_2}
\tfrac{m}{\tau}u'(\chi) + h'(\vv(\chi))\tfrac{m}{\tau}\vv'(\chi) = \tfrac{1}{\tau}\prn{U(\vv(\chi))-u(\chi)}\;,
\end{equation}
and hence the jamiton ODE \eqref{eq:jamiton_ode} carries over, with the only modification that the function $r(\vv)$ now reads as
\begin{equation}
\label{eq:definition_r}
r(\vv) = mh(\vv)+m^2\vv\;.
\end{equation}
Thus, the condition for the vanishing of the denominator in \eqref{eq:jamiton_ode} is now $h'(\vS) = -m$, and therefore for a given sonic value $\vS$, we obtain
\begin{equation}
\label{eq:values_m_s_aw_rascle_zhang_model}
m = -h'(\vS) \quad\text{and}\quad s = U(\vS)-m\vS\;.
\end{equation}
For shocks in the ARZ model, the Rankine-Hugoniot conditions \eqref{eq:aw_rascle_zhang_model_rankine_hugoniot_conditions} yield
\begin{equation*}
m\brk{\vv} + \brk{h(\vv)} = 0\;.
\end{equation*}
Thus, with $r(\vv)$ defined as in \eqref{eq:definition_r}, the same condition $r(\vv^-) = r(\vv^+)$ as in the PW model applies. The arguments involving the Lax entropy condition transfer exactly from the PW model, as do the considerations regarding roots of $w(\vv)$, as well as the jamiton construction principle. We can therefore formulate
\begin{thm}
\label{thm:equivalence_jamitons_instability_aw_rascle_zhang_model}
For the ARZ model, the existence of jamitons with a sonic value $\vS$ is equivalent (with the exception of the degenerate case, see Rem.~\ref{rem:degenerate_SCC}) to the instability of the base state solution $\vv(x,t) = \vS$, $u(x,t) = U(\vS)$.
\end{thm}
\begin{proof}
The condition for the existence of jamitons is that $U'(\vS)>m$, where $m = -h'(\vS)$. This condition is the complement of the stability condition \eqref{eq:scc_aw_rascle_zhang}. The remaining arguments are the same as in the proof of Thm.~\ref{thm:equivalence_jamitons_instability_payne_whitham_model}.
\end{proof}
Note that for the ARZ model, the connection between the existence of traveling wave solutions and the instability of base states was observed by Greenberg \cite{Greenberg2004}, albeit without drawing connections to ZND theory, using the SCC, or exploiting the geometric interpretations thereof.

\vspace{1.5em}
\section{Jamitons in the Fundamental Diagram}
\label{sec:jamitons_fundamental_diagram}
We now use the explicit constructions of jamiton solutions of the PW model and the ARZ model, presented in Sect.~\ref{sec:jamiton_construction}, to generate set-valued fundamental diagrams of traffic flow. To do so, we consider the solutions of \eqref{eq:payne_whitham_model} or \eqref{eq:aw_rascle_zhang_model}, when initializing the model with constant base state initial conditions $(\rho,U(\rho))$ plus a small perturbation. For each value of $\rho$, for which condition \eqref{eq:scc_payne_whitham} or \eqref{eq:scc_aw_rascle_zhang}, respectively, is satisfied, the solution evolves towards the constant state with density $\rho$. In contrast, for each value of $\rho$, for which the SCC is violated, the second order model \eqref{eq:payne_whitham_model} or \eqref{eq:aw_rascle_zhang_model}, respectively, converges to a jamiton-dominated state. As explained in Sect.~\ref{sec:stability}, the former statement follows from known rigorous results for second order traffic flow models with relaxation terms, while the latter statement is based on numerical experiments.

In order to construct a fundamental diagram, we start with a plot of the equilibrium curve of traffic flow
\begin{equation*}
Q_\text{eq}(\rho) = \rho U(\rho)
\end{equation*}
in the flow rate vs.\ density plane. For every value of $\rho$ that satisfies the SCC, we can record a point $(\rho,Q_\text{eq}(\rho))$ on this equilibrium curve, representing the stable base state solution. In turn, for every value of $\rho$ that violates the SCC, we sketch a jamiton line segment, which is defined as follows.

\begin{figure}
\begin{minipage}[b]{.485\textwidth}
\includegraphics[width=\textwidth]{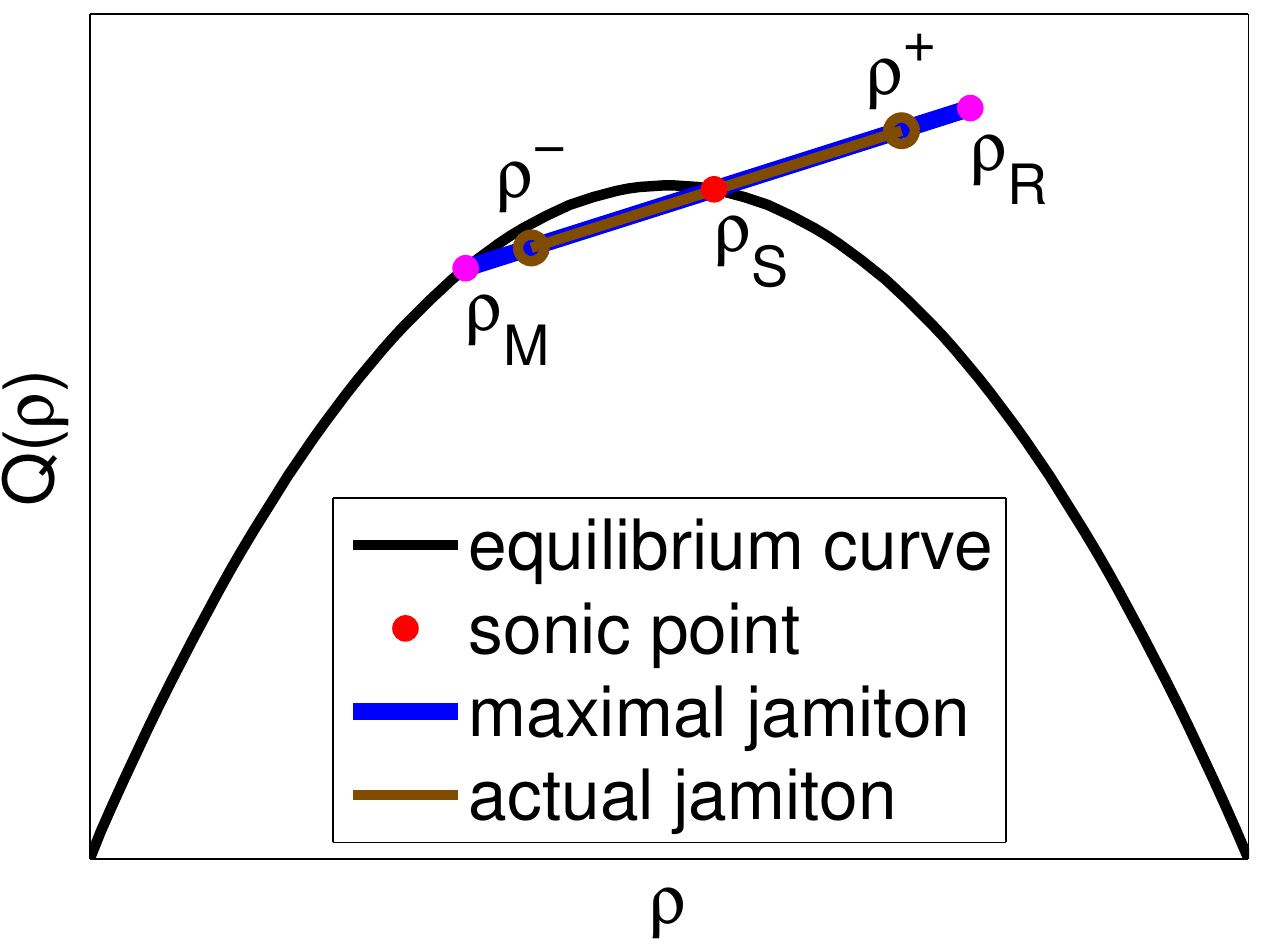}
\vspace{-1.8em}
\caption{The maximal jamiton (blue line) and an actual jamiton (brown line) in the $Q$--$\rho$ plane. Both jamitons lie on the same straight line that goes through the sonic point $(\rho_\text{S},Q(\rho_\text{S}))$.}
\label{fig:qrho_diagram}
\end{minipage}
\hfill
\begin{minipage}[b]{.485\textwidth}
\includegraphics[width=\textwidth]{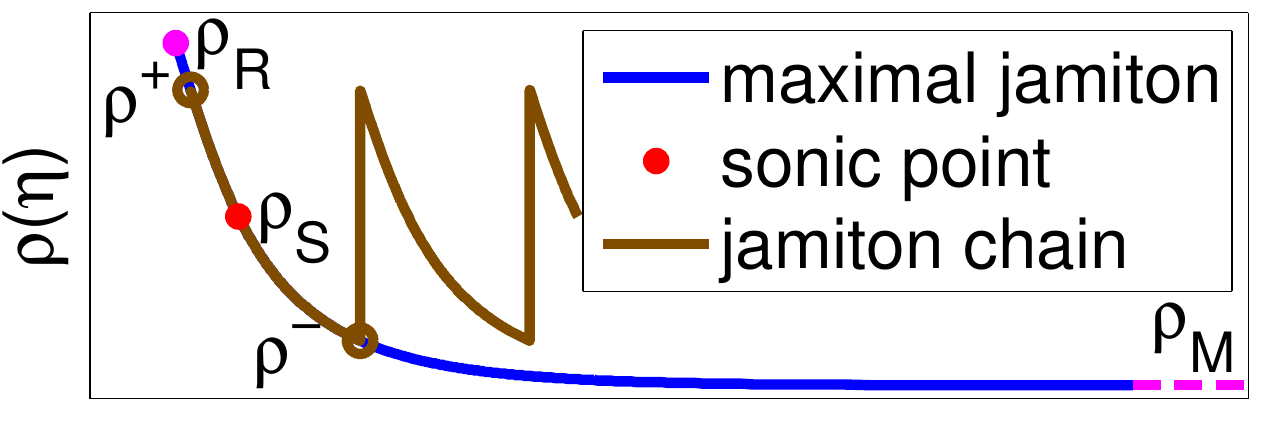} \\[.08em]
\includegraphics[width=\textwidth]{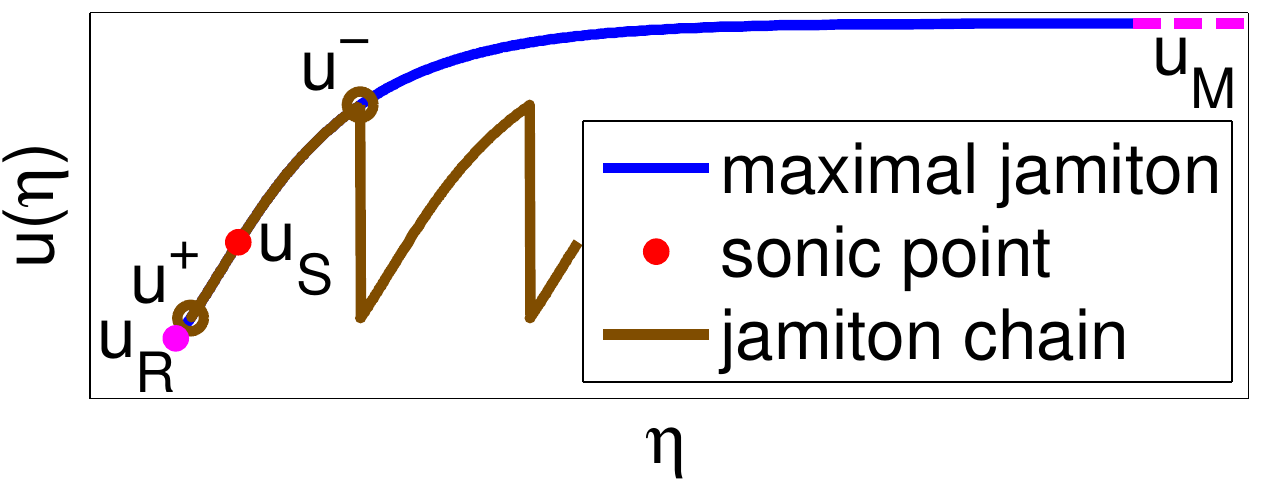}
\vspace{-1.8em}
\caption{Jamiton profiles in Eulerian variables: density $\rho(\eta)$ (top) and velocity $u(\eta)$ (bottom). As in Fig.~\ref{fig:jamiton_curve}, the periodic chain of jamitons (brown) follows the maximal jamiton curve (blue).}
\label{fig:jamiton_curve_rho_u}
\end{minipage}
\end{figure}

As derived in Sect.~\ref{sec:jamiton_construction}, a jamiton solution is non-constant in $\rho$ and $u$, and therefore it does not represent a single point in the flow rate vs.\ density plane. Instead, due to relation \eqref{eq:relation_v_u}, which implies that $u = \frac{m}{\rho}+s$, each jamiton is given by a segment of the line
\begin{equation*}
Q = m+s\rho
\end{equation*}
in the flow rate vs.\ density plane. An example of such a jamiton line is depicted in Fig.~\ref{fig:qrho_diagram}, with the corresponding jamiton profiles in Eulerian variables ($\rho(\eta)$ and $u(\eta)$) depicted in Fig.~\ref{fig:jamiton_curve_rho_u}. Due to the Chapman-Jouguet condition \eqref{eq:chapman_jouguet_condition}, every jamiton goes through the equilibrium curve $Q_\text{eq}(\rho)$ at the sonic density $\rho_\text{S}$. For each choice of sonic density $\rho_\text{S}$ (that violates the SCC), the corresponding maximal jamiton is represented by a line segment between $(\rho_\text{M},m+s\rho_\text{M})$ and $(\rho_\text{R},m+s\rho_\text{R})$, where $m$, $s$, $\rho_\text{M} = 1/\vv_\text{M}$, and $\rho_\text{R} = 1/\vv_\text{R}$ are constructed in the first 6 steps presented in Sect.~\ref{subsubsec:jamiton_construction}. Since $w(\vv_\text{M}) = 0$, where $w(\vv)$ is defined by \eqref{eq:definition_w_r}, it follows that $\rho_\text{M} U(\rho_\text{M}) = m+s\rho_\text{M}$, i.e., the point $(\rho_\text{M},m+s\rho_\text{M})$ lies on the equilibrium curve $Q_\text{eq}(\rho)$ as well. Because $Q_\text{eq}(\rho)$ is assumed concave, the maximal jamiton line therefore lies below the equilibrium curve for densities $\rho_\text{M}<\rho<\rho_\text{S}$, and above the equilibrium curve for densities $\rho_\text{S}<\rho\le\rho_\text{R}$.

Moreover, the slope of each jamiton line, $s$, equals the propagation speed of the jamiton in a stationary frame of reference. In particular, if $s>0$, the jamiton is moving forward, and if $s<0$, the jamiton is moving backward, relative to a stationary observer. Here, the former case is associated with light traffic and the latter case with heavy traffic, as formalized by the following theorem, which provides information about the structural properties of a jamiton FD.
\begin{thm}
\label{thm:properties_jamiton_FD}
For both the PW and the ARZ model, the following two properties hold:
\begin{enumerate}[(i)]
\item
The jamiton velocity $s$ decreases with the sonic density $\rho_\text{S}$.
\item
At the transition points between stability and jamiton solutions, the jamiton lines are parallel to the equilibrium curve.
\end{enumerate}
\end{thm}
\begin{proof}
For (i), we show equivalently that $s$ increases with the sonic specific volume $\vS$. Due to Thm.~\ref{thm:equivalence_jamitons_instability_payne_whitham_model} and Thm.~\ref{thm:equivalence_jamitons_instability_aw_rascle_zhang_model}, jamitons exist when $U'(\vS)-m(\vS)>0$. Due to \eqref{eq:values_m_s_payne_whitham_model} and \eqref{eq:values_m_s_aw_rascle_zhang_model}, we have that $s(\vS) = U(\vS)-m(\vS)\vS$. Hence
\begin{equation*}
s'(\vS) = U'(\vS)-m(\vS)-m'(\vS)\vS > -m'(\vS)\vS\;.
\end{equation*}
Moreover, since we assume that $p''(\vv)>0$ for the PW model, and $h''(\vv)>0$ for the ARZ model, we find that $m'(\vS)<0$ and therefore $s'(\vS)>0$.

To show (ii), we note that due to Thm.~\ref{thm:equivalence_jamitons_instability_payne_whitham_model} and Thm.~\ref{thm:equivalence_jamitons_instability_aw_rascle_zhang_model}, the boundary between stability and jamitons occurs, in Eulerian variables, when $-U'(\rho_\text{S})\rho_\text{S}^2 = m(\rho_\text{S})$. Hence
\begin{equation*}
s(\rho_\text{S}) = U(\rho_\text{S})-\tfrac{m(\rho_\text{S})}{\rho_\text{S}}
= U(\rho_\text{S})+U'(\rho_\text{S})\rho_\text{S}
= Q'(\rho_\text{S})\;,
\end{equation*}
which proves the claim.
\end{proof}

\begin{rem}
It should be emphasized that jamitons can travel forward ($s>0$) or backward ($s<0$) on the road, or be stationary ($s=0$). As Thm.~\ref{thm:properties_jamiton_FD} states, the larger the traffic density, the smaller $s$. An extreme case is the one of true stop-and-go waves, in which vehicles come to a complete stop. These waves, visible for instance in \cite{VideoTrafficWaves}, always travel backwards on the road. However, waves with $s>0$ can also be observed, albeit less pronounced than true stop-and-go waves. Instances of such forward traveling jamitons can be seen in the NGSIM trajectory data \cite{TrafficNGSIM} of the I-80 during the interval 4:00pm--4:15pm.
\end{rem}

\subsection{List of Examples}
\label{subsec:examples}
Before presenting the construction of jamiton FDs (and extensions thereof, see below) in detail, we first collect the parameters of the examples that we shall use as test cases. In all examples, we consider $\rho_\text{max} = 1/7.5\text{m}$ and $u_\text{max} = 20\text{m}/\text{s}$, except when explicitly noted otherwise. Specifically, we consider the following four scenarios:
\begin{itemize}
\item[\textbf{PW1}]
The PW model with a quadratic equilibrium curve, i.e., a linear desired velocity function \eqref{eq:desired_velocity_linear}, and a singular pressure of the form \eqref{eq:logarithmic_pressure}, with $\beta = 4.8\text{m}/\text{s}^2$. This model is used in Figs.~\ref{fig:fd_pw_uquad_plog_00}--\ref{fig:fd_pw_uquad_plog_inf}.
\item[\textbf{PW2}]
The PW model with a non-quadratic concave equilibrium curve of the form
\begin{equation}
\label{eq:smoothed_Newell_Daganzo_flux}
Q(\rho) = c\prn{g(0)+(g(1)-g(0))\rho/\rho_\text{max}
-g(\rho/\rho_\text{max})}\;,
\end{equation}
where the choice of the function $g(y) = \sqrt{1+(\tfrac{y-b}{\lambda})^2}$ has the effect that \eqref{eq:smoothed_Newell_Daganzo_flux} can be interpreted as a smooth version of the piecewise-linear Newell-Daganzo flux \cite{Newell1993, Daganzo1994}. The parameters are chosen such that a reasonable fit with the measurement data, shown in Fig.~\ref{fig:fd_data}, is achieved: $c = 0.078\rho_\text{max}u_\text{max}$, $b =
\frac{1}{3}$, and $\lambda = \frac{1}{10}$. Again, the logarithmic pressure \eqref{eq:logarithmic_pressure} is chosen, this time with $\beta = 8\text{m}/\text{s}^2$. This model is used in Figs.~\ref{fig:fd_pw_ukink_plog_00}--\ref{fig:fd_pw_ukink_plog_inf}.
\item[\textbf{ARZ1}]
The ARZ model with the same equilibrium curve \eqref{eq:smoothed_Newell_Daganzo_flux}, and with the singular hesitation function \eqref{eq:BDDR_hesitation}, which was suggested in \cite{BerthelinDegondDelitalaRascle2008}, with $\gamma = \frac{1}{2}$ and $\beta = 8\text{m}/\text{s}^2$. This model is used in Figs.~\ref{fig:fd_ar_ukink_pbddr_00}--\ref{fig:fd_ar_ukink_pbddr_inf}.
\item[\textbf{ARZ2}]
The ARZ model with the equilibrium curve \eqref{eq:smoothed_Newell_Daganzo_flux}, and with the singular hesitation function
\begin{equation*}
h(\rho) = \beta\frac{\prn{\rho/\rho_\text{max}}^{\gamma_1}}
{\prn{1-\rho/\rho_\text{max}}^{\gamma_2}}\;.
\end{equation*}
where $\gamma_1 = \frac{1}{5}$, $\gamma_2 = \frac{1}{10}$, and $\beta = 12\text{m}/\text{s}^2$. This generalization of \eqref{eq:BDDR_hesitation} yields a better qualitative fit of the unstable regime with measurement data (see Sect.~\ref{sec:mimicking_measurement_data}). This model is used in Fig.~\ref{fig:fd_ar_with_data}.
\end{itemize}
Note that some of the model parameters ($\rho_\text{max}$, $u_\text{max}$) are chosen to be within ranges of values that are realistic for real traffic, some ($c$, $b$, $\lambda$) to match specific data, and some ($\beta$, $\gamma_i$) to yield reasonably sized jamiton regions.

\subsection{Maximal Jamiton Envelopes}
\label{subsec:fundamental_diagram_maximal_jamiton}
The relations described in the beginning of this section give rise to the following construction of a set-valued fundamental diagram for the PW or the ARZ model, defined by a desired velocity function $U(\rho)$ and the traffic pressure $p(\rho)$ or the hesitation function $h(\rho)$, respectively.

Loop over all densities $0\le \rho \le \rho_\text{max}$. For each density $\rho$:
\begin{enumerate}[\quad 1.]
\item Check whether the SCC is satisfied. If it is, draw the equilibrium point $(\rho,Q_\text{eq}(\rho))$.
\item If it is not, compute $m$, $s$, $\rho_\text{M}$, and $\rho_\text{R}$ (see Sect.~\ref{subsubsec:jamiton_construction}), and plot a straight line between the points $(\rho_\text{M},m+s\rho_\text{M})$ and $(\rho_\text{R},m+s\rho_\text{R})$.
\end{enumerate}
An example of such a \emph{jamiton fundamental diagram} is shown in Fig.~\ref{fig:fd_pw_uquad_plog_00}. It consists of segments of the equilibrium curve (wherever $\rho$ satisfies the SCC; shown in black), and a set-valued region in the flow rate vs.\ density plane, in which the equilibrium curve lies (for all $\rho$ that violate the SCC). We call this set-valued region the \emph{jamiton region}, which is formed by a continuum of jamiton line segments, shown in blue in all figures that present a jamiton FD. The boundaries of the jamiton region are formed by an upper envelope curve above the equilibrium curve and a lower envelope curve below the equilibrium curve, both of which are depicted in magenta.

Theorem~\ref{thm:properties_jamiton_FD}, together with the concavity of $Q(\rho)$, implies that different jamiton lines cannot intersect above the equilibrium curve. Hence, the upper envelope of the jamiton region is a curve that is formed by the points $(\rho_\text{R},m+s\rho_\text{R})$, where $m$, $s$, and $\rho_\text{R}$ are parameterized by $\rho_\text{S}$. In contrast, different jamiton lines do intersect below the equilibrium curve. Therefore, the lower envelope of the jamiton region is formed by intersection points of jamiton lines. In Eulerian variables, this curve, wherever it is below the equilibrium curve, is given by $(\rho^*,Q^*)$, where
\begin{equation*}
\rho^*(\rho_\text{S}) = -\frac{m'(\rho_\text{S})}{s'(\rho_\text{S})}
\quad\text{and}\quad
Q^*(\rho_\text{S}) = m(\rho_\text{S})+s(\rho_\text{S})\rho^*(\rho_\text{S})
\end{equation*}
represent the parametrization by $\rho_\text{S}$.

This construction of a maximal jamiton FD is applied to the test cases PW1, PW2, and AR1, described in Sect.~\ref{subsec:examples}. The results are shown in Figs.~\ref{fig:fd_pw_uquad_plog_00},~\ref{fig:fd_pw_ukink_plog_00}, and~\ref{fig:fd_ar_ukink_pbddr_00}. One can clearly observe the structural properties described above. In comparison with the FD obtained from measurement data (see Fig.~\ref{fig:fd_data}), one particular disagreement is apparent: in the jamiton FD, the upper envelope curve extends much further than it is observed in the data FD. This disagreement can both be explained and remedied, as done below.

\begin{figure}
\begin{minipage}[b]{.485\textwidth}
\includegraphics[width=\textwidth]{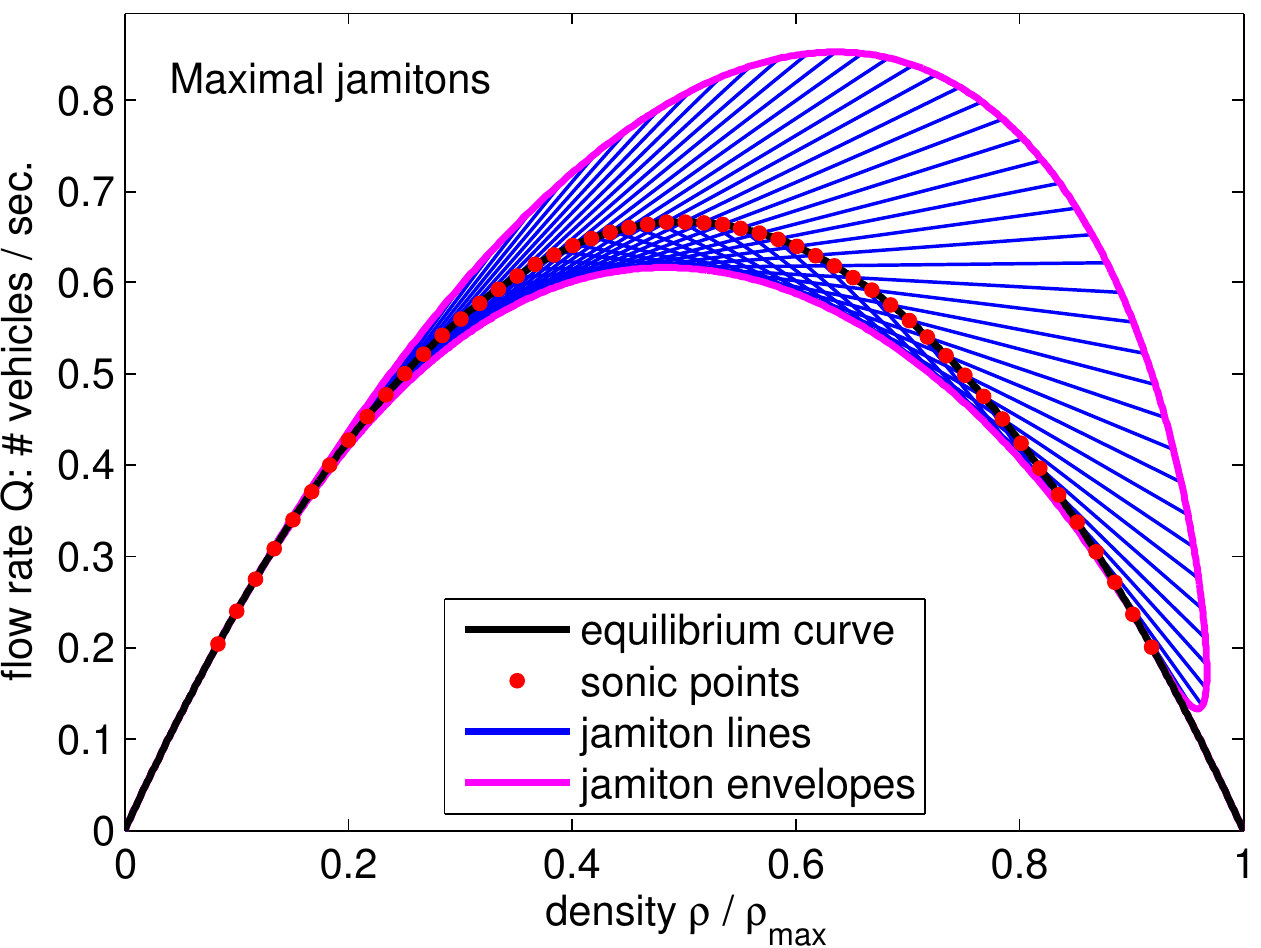}
\vspace{-1.8em}
\caption{Maximal jamiton FD (i.e., no temporal aggregation) for the example \textbf{PW1} (see Sect.~\ref{subsec:examples}).}
\label{fig:fd_pw_uquad_plog_00}
\end{minipage}
\hfill
\begin{minipage}[b]{.485\textwidth}
\includegraphics[width=\textwidth]{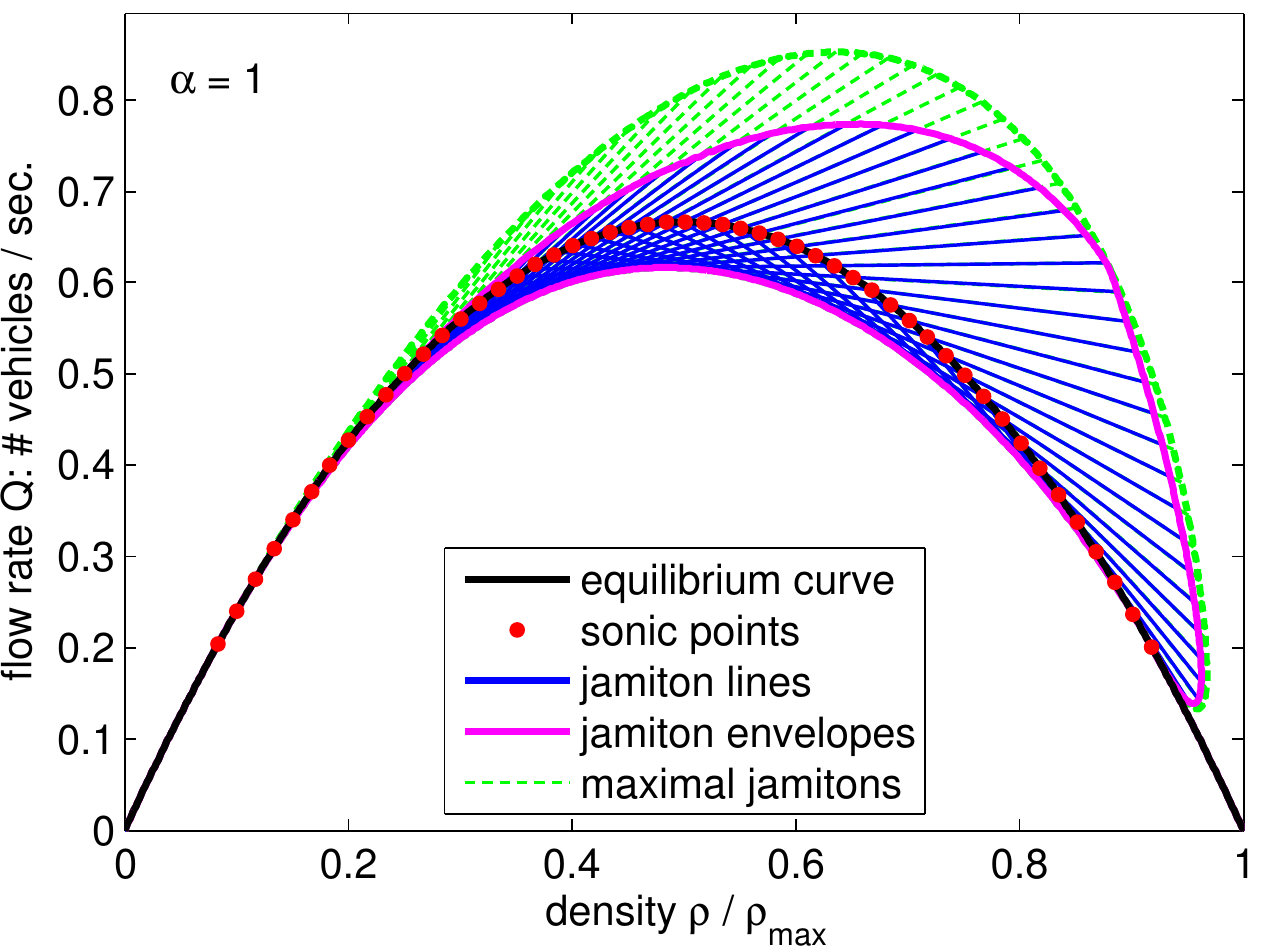}
\vspace{-1.8em}
\caption{Jamiton FD with short aggregation time $\Delta t = \tau$ for the example \textbf{PW1} (see Sect.~\ref{subsec:examples}). \\}
\label{fig:fd_pw_uquad_plog_01}
\end{minipage}

\vspace{.5em}
\begin{minipage}[b]{.485\textwidth}
\includegraphics[width=\textwidth]{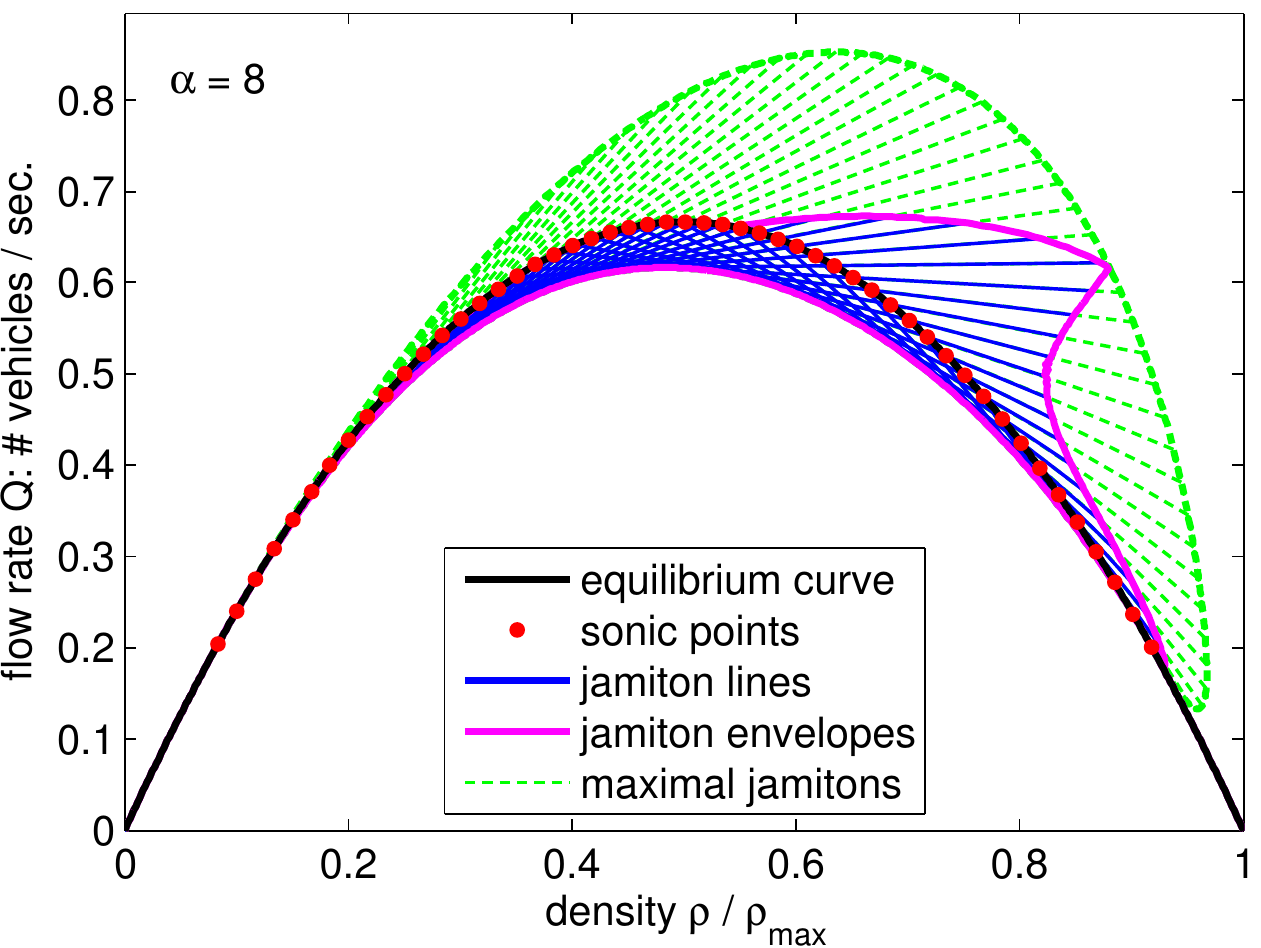}
\vspace{-1.8em}
\caption{Jamiton FD with long aggregation time $\Delta t = 8\tau$ for the example \textbf{PW1} (see Sect.~\ref{subsec:examples}).}
\label{fig:fd_pw_uquad_plog_08}
\end{minipage}
\hfill
\begin{minipage}[b]{.485\textwidth}
\includegraphics[width=\textwidth]{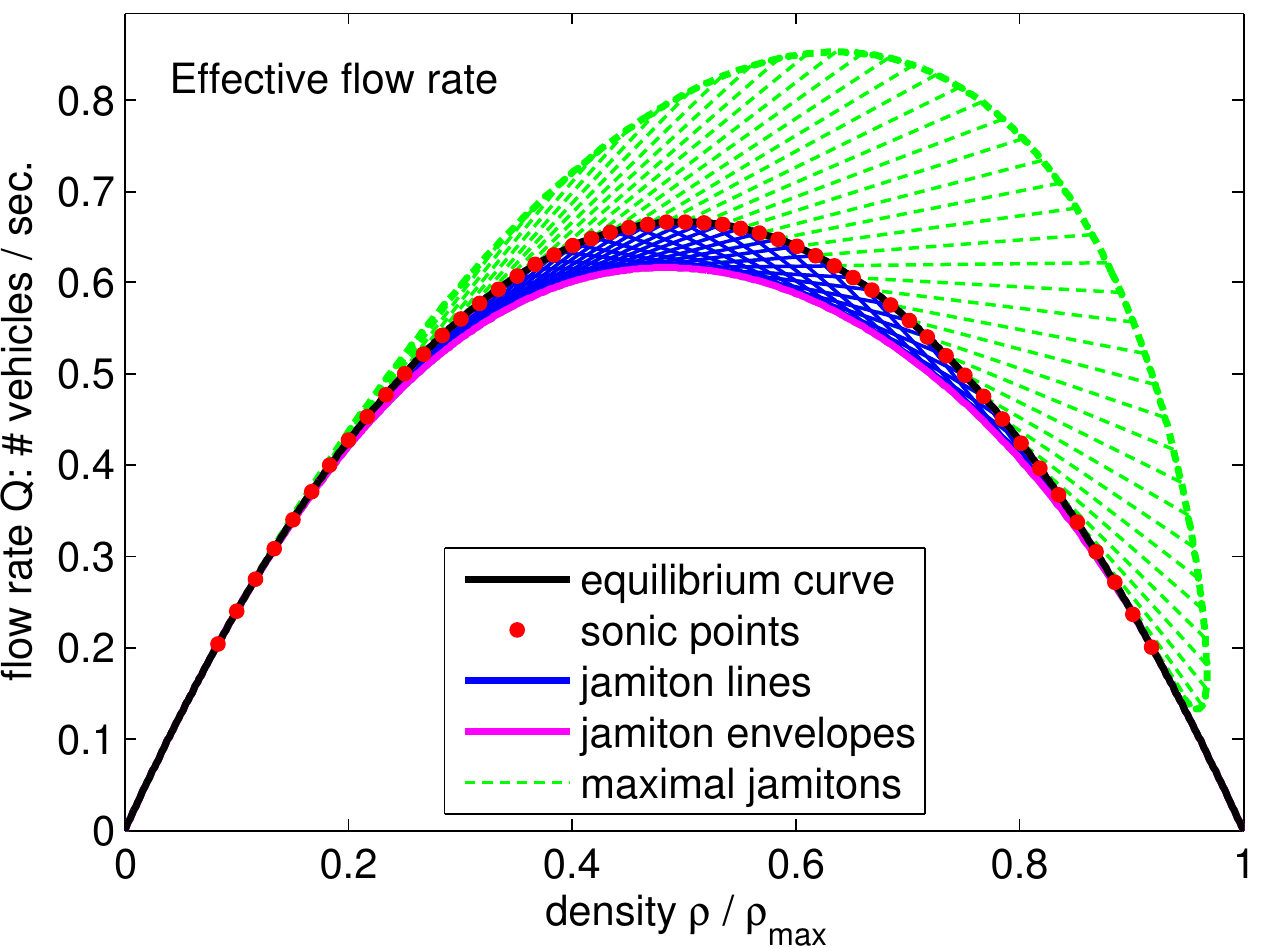}
\vspace{-1.8em}
\caption{Jamiton FD with effective flow rate (i.e., aggregation over complete jamitons) for the example \textbf{PW1} (see Sect.~\ref{subsec:examples}).}
\label{fig:fd_pw_uquad_plog_inf}
\end{minipage}
\end{figure}

\begin{figure}
\begin{minipage}[b]{.485\textwidth}
\includegraphics[width=\textwidth]{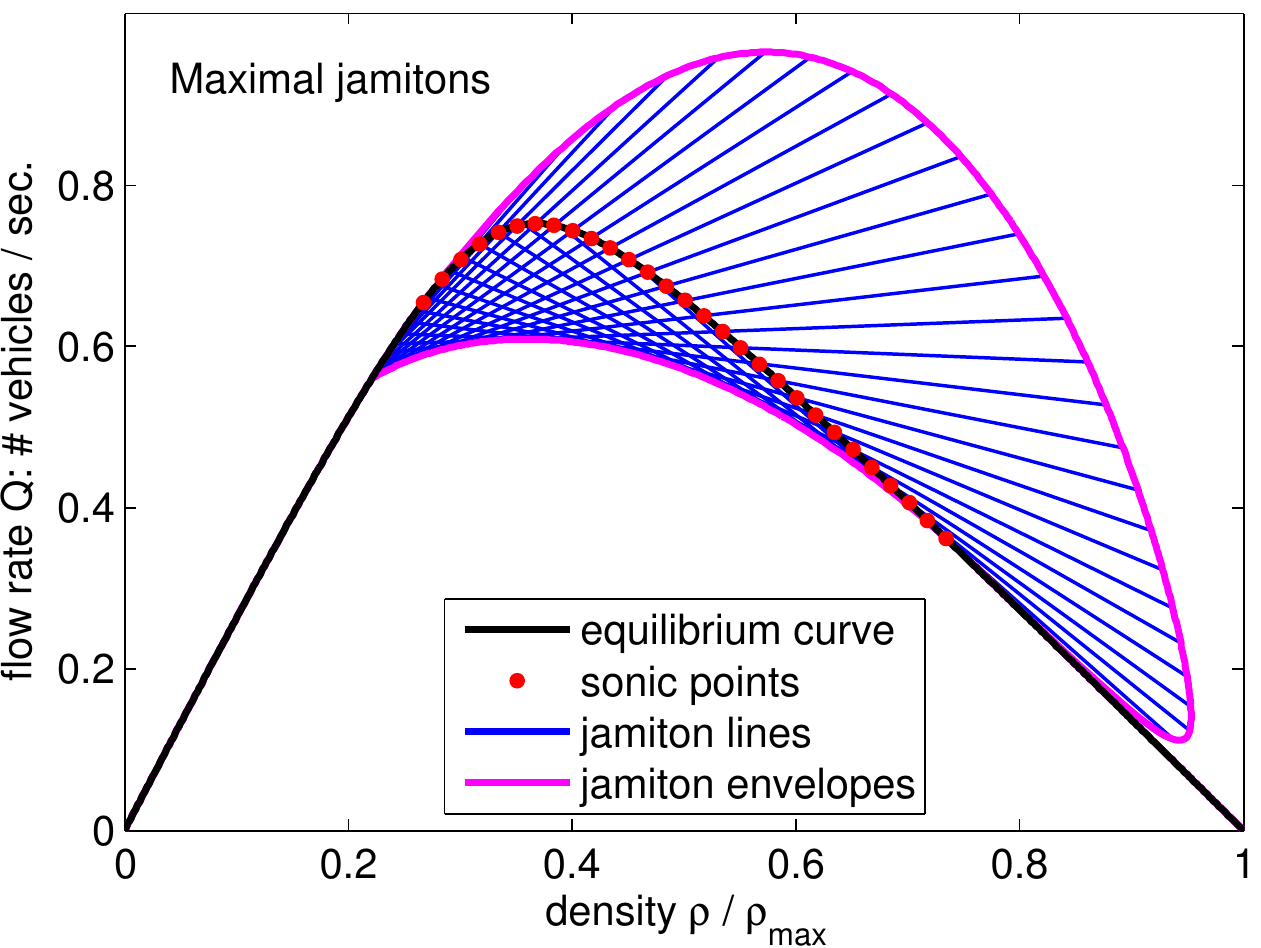}
\vspace{-1.8em}
\caption{Maximal jamiton FD (i.e., no temporal aggregation) for the example \textbf{PW2} (see Sect.~\ref{subsec:examples}).}
\label{fig:fd_pw_ukink_plog_00}
\end{minipage}
\hfill
\begin{minipage}[b]{.485\textwidth}
\includegraphics[width=\textwidth]{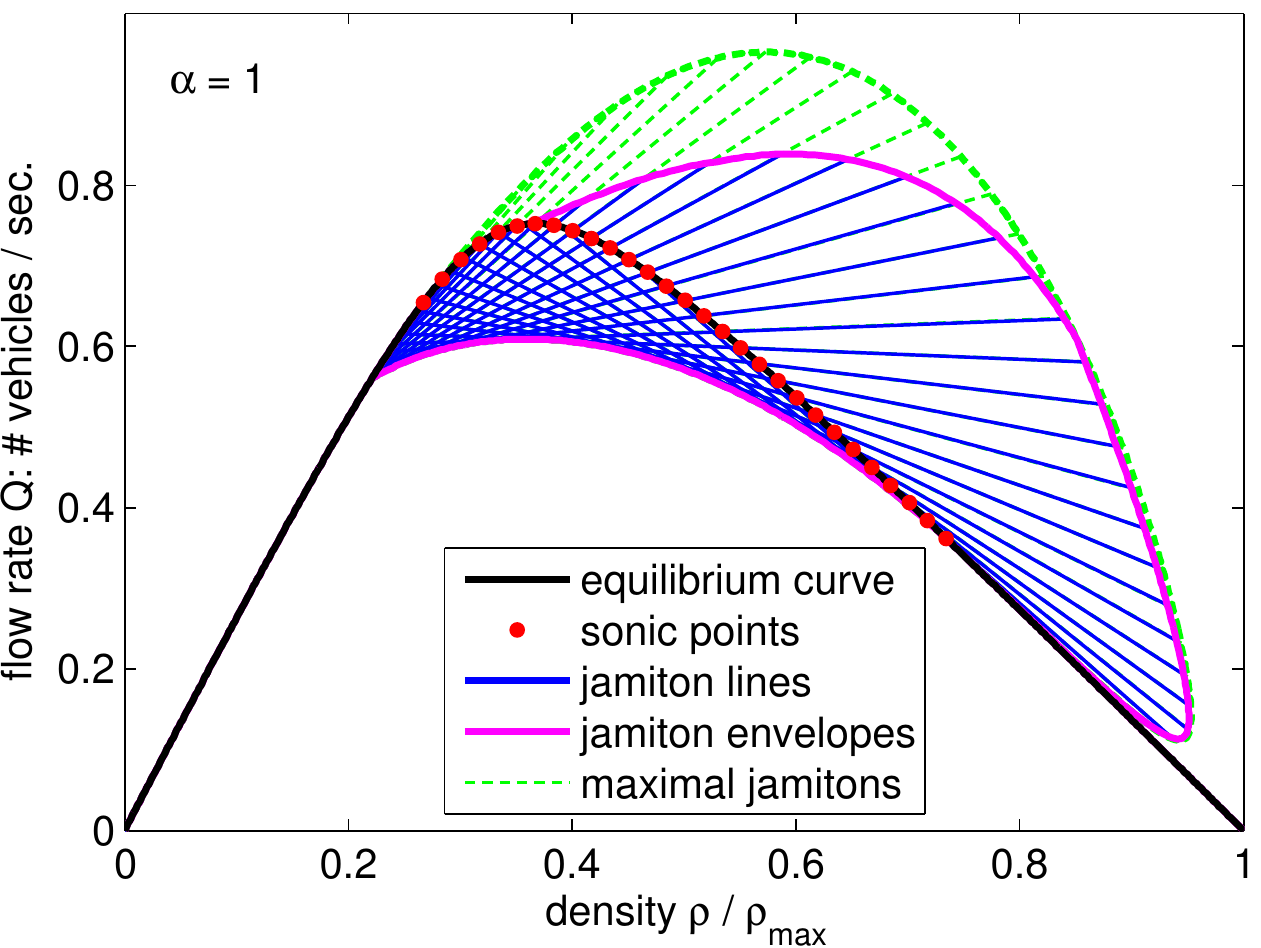}
\vspace{-1.8em}
\caption{Jamiton FD with short aggregation time $\Delta t = \tau$ for the example \textbf{PW2} (see Sect.~\ref{subsec:examples}). \\}
\label{fig:fd_pw_ukink_plog_01}
\end{minipage}

\vspace{.5em}
\begin{minipage}[b]{.485\textwidth}
\includegraphics[width=\textwidth]{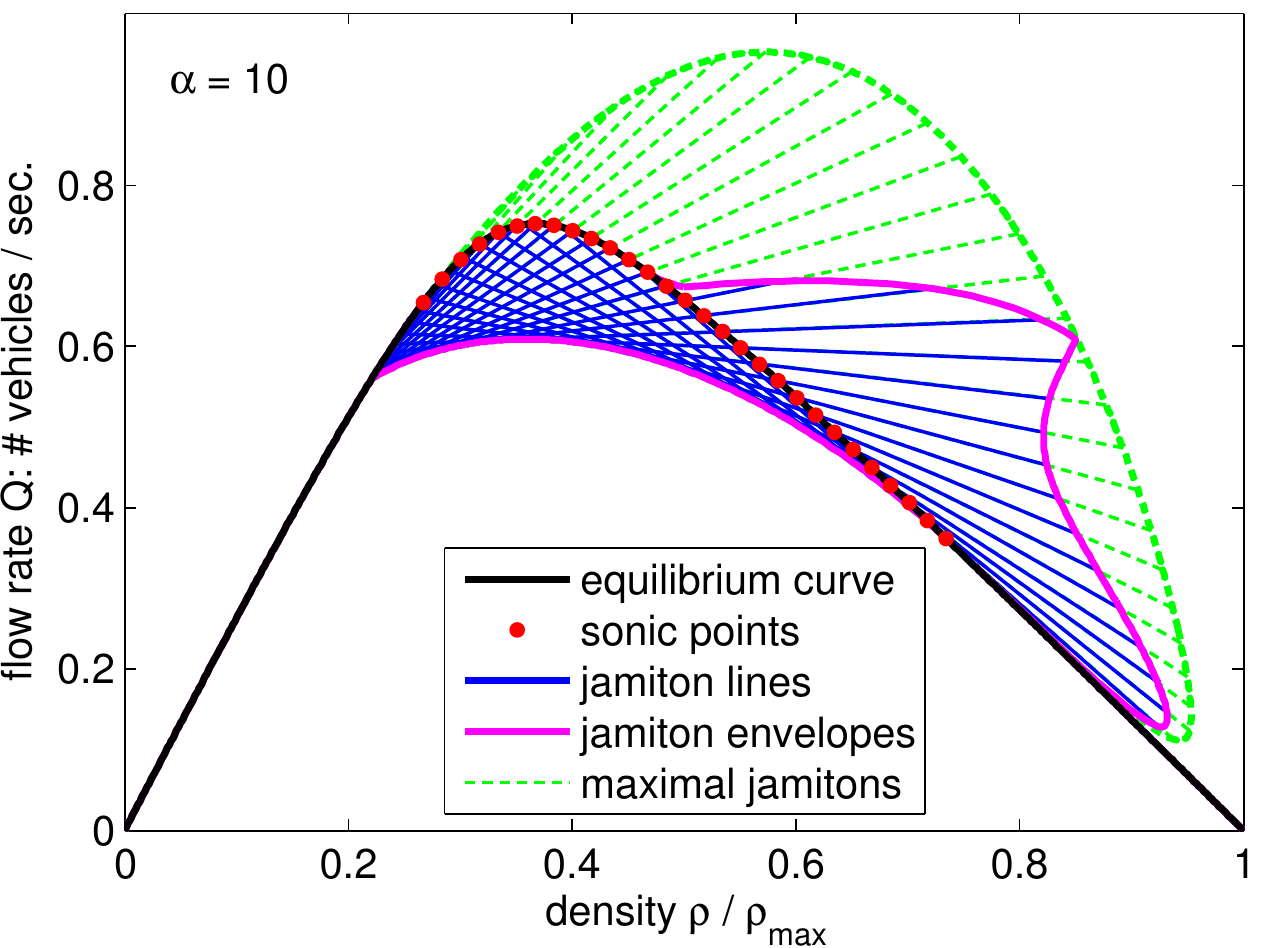}
\vspace{-1.8em}
\caption{Jamiton FD with long aggregation time $\Delta t = 10\tau$ for the example \textbf{PW2} (see Sect.~\ref{subsec:examples}).}
\label{fig:fd_pw_ukink_plog_10}
\end{minipage}
\hfill
\begin{minipage}[b]{.485\textwidth}
\includegraphics[width=\textwidth]{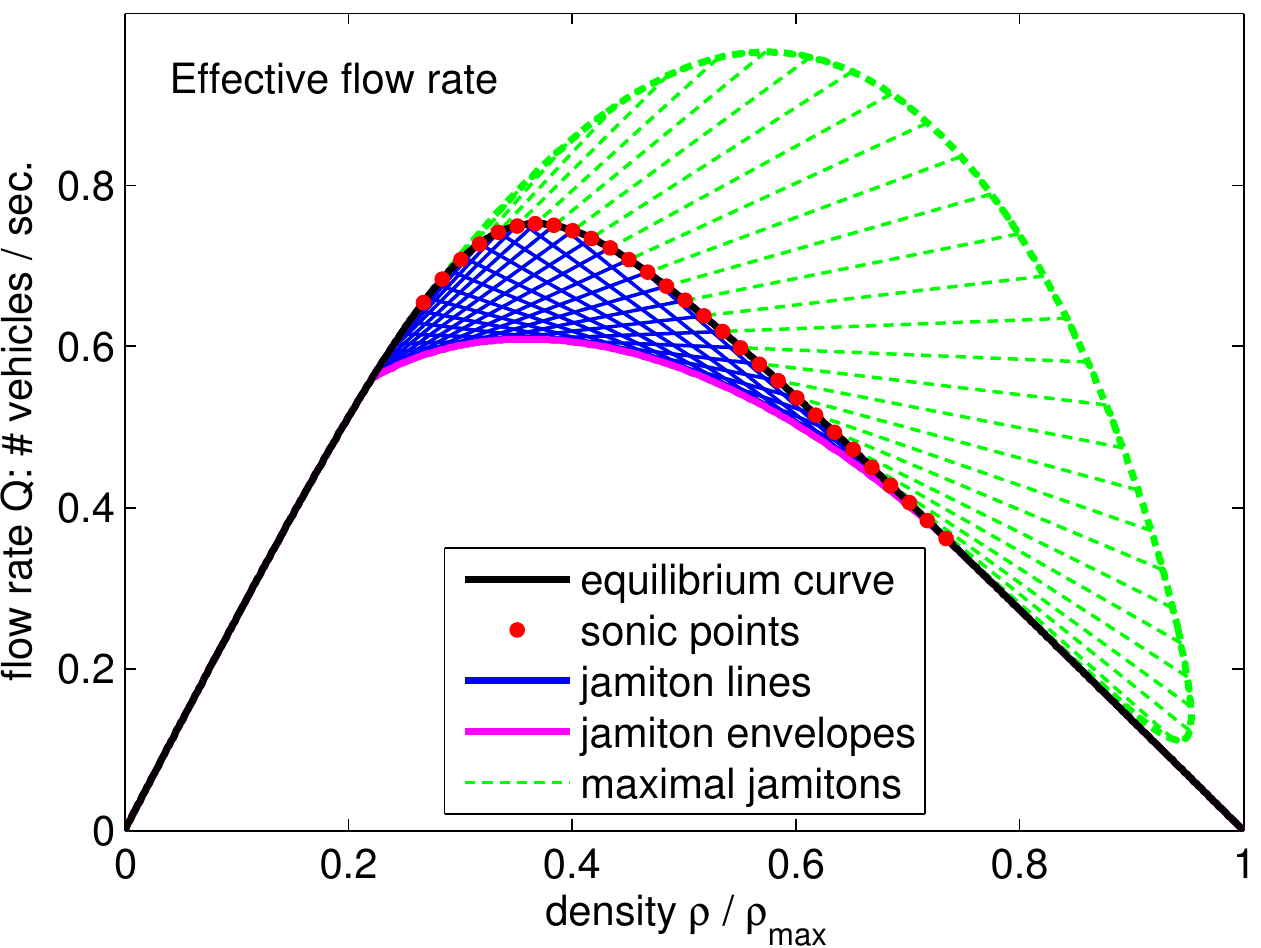}
\vspace{-1.8em}
\caption{Jamiton FD with effective flow rate (i.e., aggregation over complete jamitons) for the example \textbf{PW2} (see Sect.~\ref{subsec:examples}).}
\label{fig:fd_pw_ukink_plog_inf}
\end{minipage}
\end{figure}

\begin{figure}
\begin{minipage}[b]{.485\textwidth}
\includegraphics[width=\textwidth]{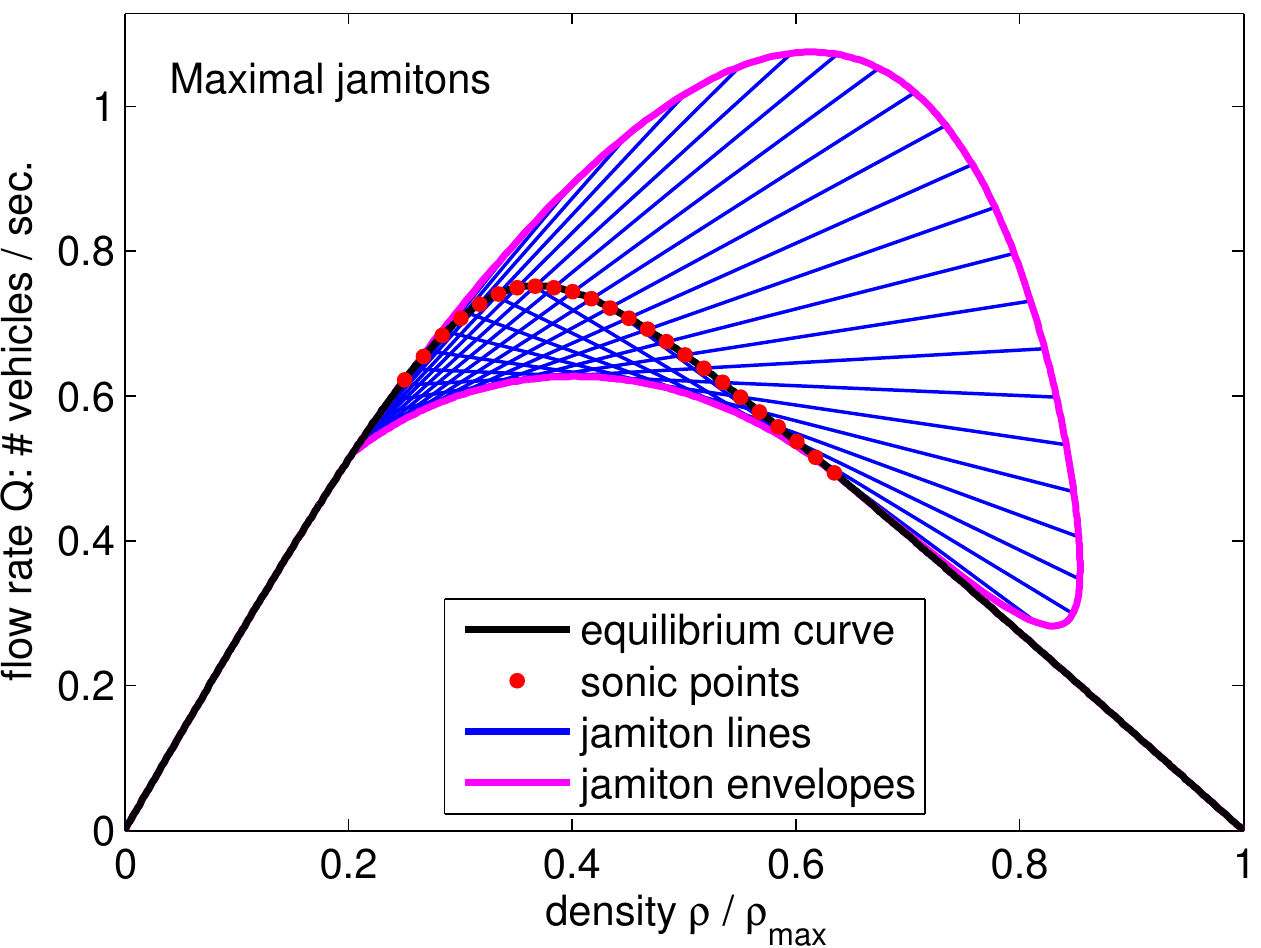}
\vspace{-1.8em}
\caption{Maximal jamiton FD (i.e., no temporal aggregation) for the example \textbf{ARZ1} (see Sect.~\ref{subsec:examples}).}
\label{fig:fd_ar_ukink_pbddr_00}
\end{minipage}
\hfill
\begin{minipage}[b]{.485\textwidth}
\includegraphics[width=\textwidth]{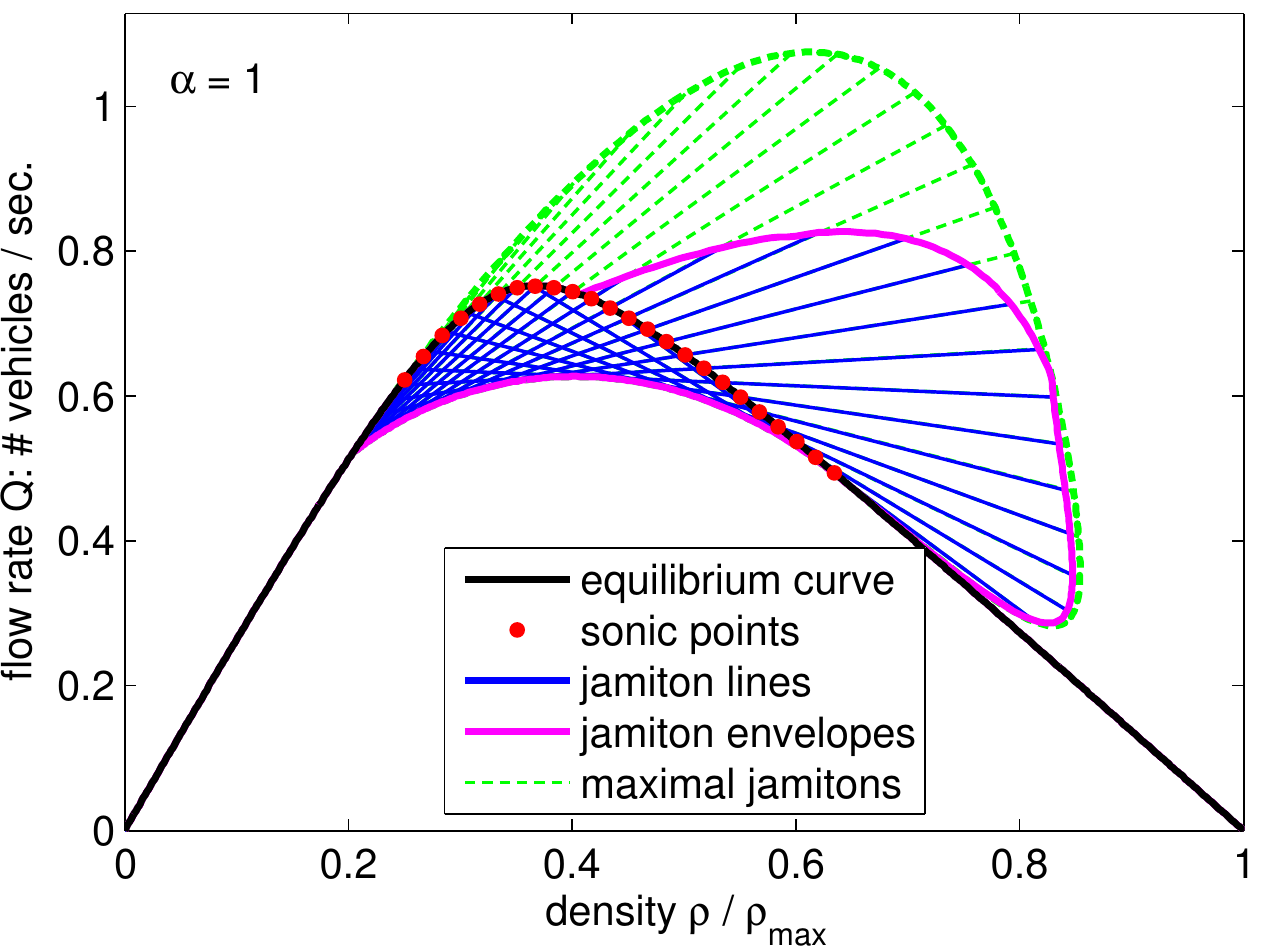}
\vspace{-1.8em}
\caption{Jamiton FD with short aggregation time $\Delta t = \tau$ for the example \textbf{ARZ1} (see Sect.~\ref{subsec:examples}).}
\label{fig:fd_ar_ukink_pbddr_01}
\end{minipage}

\vspace{.5em}
\begin{minipage}[b]{.485\textwidth}
\includegraphics[width=\textwidth]{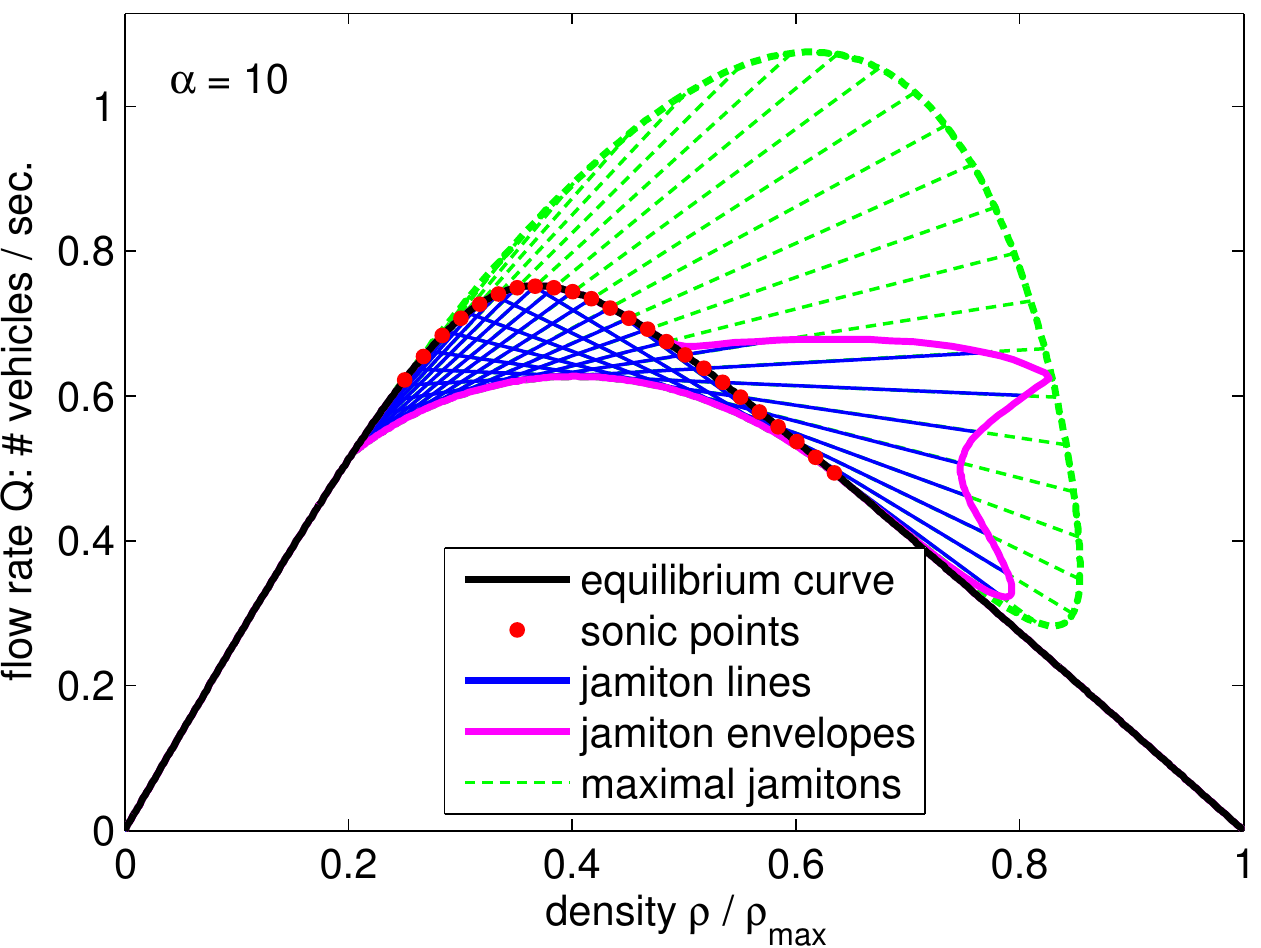}
\vspace{-1.8em}
\caption{Jamiton FD with long aggregation time $\Delta t = 10\tau$ for the example \textbf{ARZ1} (see Sect.~\ref{subsec:examples}).}
\label{fig:fd_ar_ukink_pbddr_10}
\end{minipage}
\hfill
\begin{minipage}[b]{.485\textwidth}
\includegraphics[width=\textwidth]{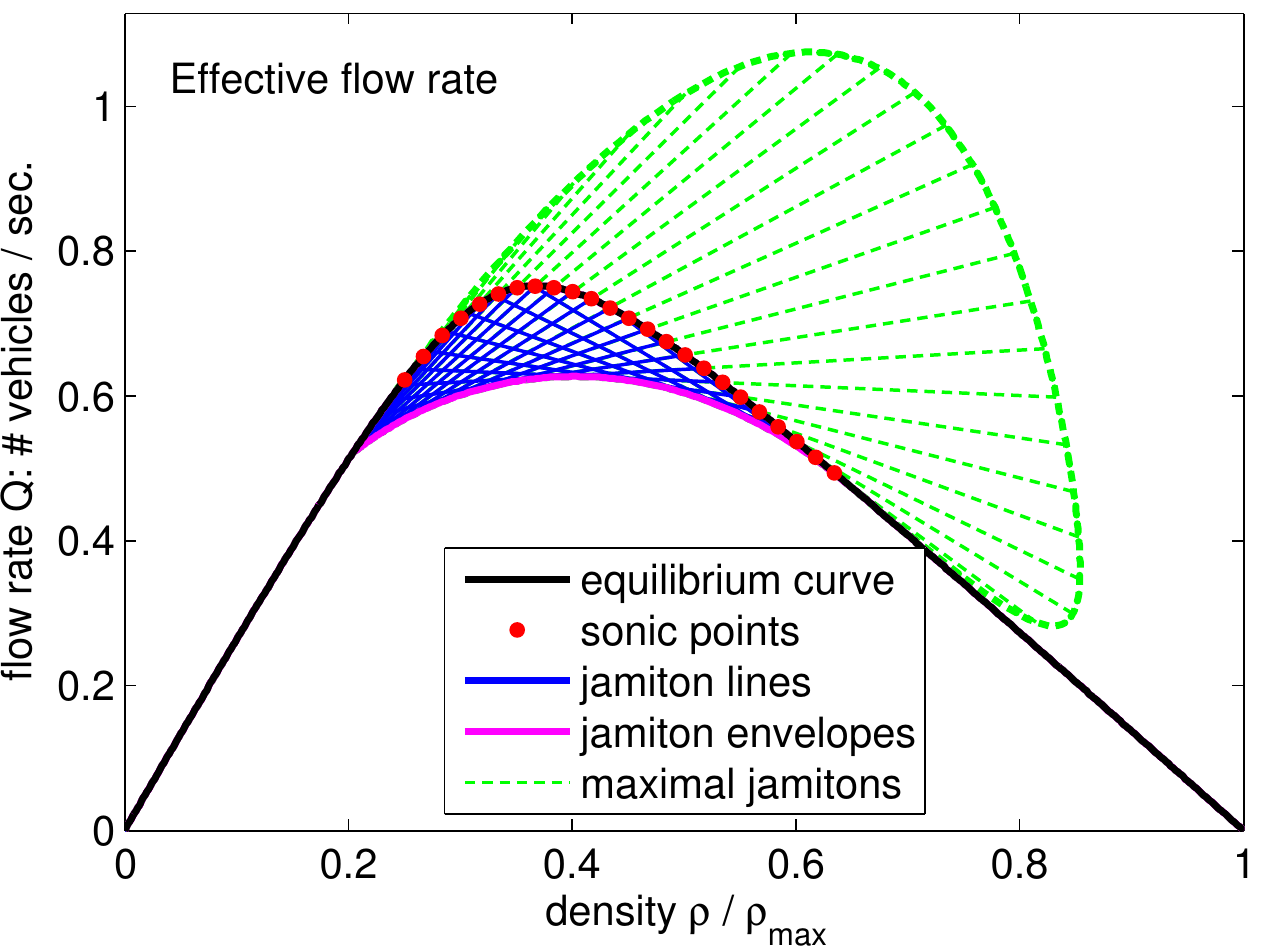}
\vspace{-1.8em}
\caption{Jamiton FD with effective flow rate (i.e., aggregation over complete jamitons) for the example \textbf{ARZ1} (see Sect.~\ref{subsec:examples}).}
\label{fig:fd_ar_ukink_pbddr_inf}
\end{minipage}
\end{figure}

\subsection{Emulating Sensor Measurements}
\label{subsec:emulating_sensor_measurements}
The jamiton FDs shown in Figs.~\ref{fig:fd_pw_uquad_plog_00},~\ref{fig:fd_pw_ukink_plog_00}, and~\ref{fig:fd_ar_ukink_pbddr_00} appear to not resemble the sensor data (shown in Fig.~\ref{fig:fd_data}) well in that they can exhibit a large region above the equilibrium curve. The reason for this discrepancy is that the jamiton FD represents point-wise values of $\rho$ and $Q$ along the jamiton profile, whereas the data FD is composed of aggregated information (see below).

In fact, the mere size of the jamiton region above the equilibrium curve is somewhat misleading: for each jamiton line in the FD, the segment below the equilibrium curve corresponds to the (possibly infinitely) long low density piece of a jamiton, while the segment above the equilibrium curve corresponds to a comparatively short piece of the jamiton, at which the density is high (see Figs.~\ref{fig:jamiton_curve} and~\ref{fig:jamiton_curve_rho_u}). In addition, as one can see in the example jamiton shown in Fig.~\ref{fig:jamiton_curve_rho_u}, the high density part frequently comes also with a large negative density gradient along the jamiton curve. As a consequence, when considering aggregated data, the region above the equilibrium curve contributes much less to an average than the region below it.

Measurement data used to generate fundamental diagrams are typically obtained by stationary sensors on the road that measure the average density and average flow rate during a time interval $\Delta t$. In other words, each point in a data FD corresponds to an aggregation of the traffic situation over a non-zero time interval. In practice, in order to obtain reliable estimates for the flow rate (which is obtained by counting vehicles), the aggregation time $\Delta t$ must not be too small. A typical value is $\Delta t = 30\text{s}$, as is used for instance in the Mn/DOT RTMC data set \cite{TrafficMnDOT}, see the description in Sect.~\ref{subsec:fundamental_diagram}. Hence, in order to obtain a better agreement with sensor measurements, we must represent the jamiton FD in a way that is more faithful to the actual measurement process. We therefore perform an averaging over jamiton solutions, as follows.

For a given sonic density $\rho_\text{S}$, we consider a one-parameter family of traveling wave solutions, each being an infinite chain of identical jamitons, where the parameter is the length of each individual jamiton, measured from shock to shock. For each such periodic traveling wave solution, we consider temporal averages of duration $\Delta t$ at a fixed position $\bar{x}$ on the road. Because the chain of jamitons moves with velocity $s$, these temporal averages can be translated into averages over the Eulerian similarity variable $\eta = \frac{x-st}{\tau}$. Specifically, the average density is
\begin{equation}
\label{eq:average_rho}
\begin{split}
\bar{\rho} &= \tfrac{1}{\Delta t}\int_0^{\Delta t} \rho(\bar{x},t)\ud{t}
= \tfrac{1}{\Delta t}\int_0^{\Delta t} \rho(-\tfrac{\bar{x}-st}{\tau})\ud{t} \\
&= \tfrac{\tau}{s\Delta t}\int_{\tfrac{\bar{x}-s\Delta t}{\tau}}^{\frac{\bar{x}}{\tau}} \rho(\eta)\ud{\eta}
= \tfrac{1}{\Delta\eta}\int_{\bar{\eta}-\Delta\eta}^{\bar{\eta}} \rho(\eta)\ud{\eta}\;,
\end{split}
\end{equation}
where $\Delta\eta = \tfrac{s\Delta t}{\tau}$ is the averaging ``length'' (in the $\eta$ variable), and $\bar{\eta} = \tfrac{\bar{x}}{\tau}$ parameterizes the ``position'' (in the $\eta$ variable) of the averaging interval. Similarly, the average flow rate equals
\begin{equation}
\label{eq:average_Q}
\bar{Q} = \tfrac{1}{\Delta\eta}\int_{\bar{\eta}-\Delta\eta}^{\bar{\eta}} Q(\eta)\ud{\eta}
= \tfrac{1}{\Delta\eta}\int_{\bar{\eta}-\Delta\eta}^{\bar{\eta}} m+s\rho(\eta)\ud{\eta}
= m+s\bar{\rho}\;,
\end{equation}
which implies that any average pair $(\bar{\rho},\bar{Q})$ lies on the corresponding jamiton line $Q = m+s\rho$. Note that these averages depend only on the dimensionless quantity
\begin{equation*}
\alpha = \frac{\Delta t}{\tau}\;,
\end{equation*}
rather than $\Delta t$ or $\tau$ alone, which measures the temporal averaging duration in units of the drivers' relaxation time $\tau$.

For sufficiently large $\Delta t$ and for sufficiently short jamitons, the averages \eqref{eq:average_rho} and \eqref{eq:average_Q} may involve one or more shocks. In this case, it is helpful to split the integrals into multiple segments, each of which corresponds to a smooth jamiton curve. For each of these segments, it follows from \eqref{eq:jamiton_mass} that
\begin{equation*}
\bar{\rho}
= \tfrac{1}{\Delta\eta}\int_{\bar{\eta}-\Delta\eta}^{\bar{\eta}} \rho(\eta)\ud{\eta}
= \tfrac{1}{\Delta\eta}\int_{\vv(\bar{\eta}-\Delta\eta)}^{\vv(\bar{\eta})} \frac{r'(\vv)}{w(\vv)}\ud{\vv}\;,
\end{equation*}
which can be computed efficiently using composite Gaussian quadrature rules.

Using these averaging procedures, we can now, for each choice of $\alpha = \tfrac{\Delta t}{\tau}$, produce an \emph{aggregated jamiton fundamental diagram}, by plotting on each jamiton line segment all possible pairs $(\bar{\rho},\bar{Q})$ that can arise as averages. When doing so, the jamiton line segment above the equilibrium curve is shortened, because the peak density $\rho_\text{R}$ is averaged with lower density values. This can be observed in Fig.~\ref{fig:jamiton_curve_rho_u}: even for the maximal jamiton, the segment where $\rho>\rho_\text{S}$ is of finite length, and thus the maximum possible average $\bar{\rho}$ decreases with increasing averaging length $\Delta\eta$. In turn, the line segment below the equilibrium curve is not shortened: due to the infinite length of the maximal jamiton towards $\rho_\text{M}$ (see Fig.~\ref{fig:jamiton_curve_rho_u}), averages can always become as low as $\rho_\text{M}$ itself, no matter how large $\Delta\eta$ is chosen.

Aggregated jamiton FDs are shown for the example PW1 in Figs.~\ref{fig:fd_pw_uquad_plog_01} and~\ref{fig:fd_pw_uquad_plog_08}, for the example PW2 in Figs.~\ref{fig:fd_pw_ukink_plog_01} and~\ref{fig:fd_pw_ukink_plog_10}, and for the example AR1 in Figs.~\ref{fig:fd_ar_ukink_pbddr_01} and~\ref{fig:fd_ar_ukink_pbddr_10}. In all three examples, the first figure corresponds to a short aggregation duration of $\Delta t = \tau$, and the second figure corresponds to a long aggregation duration of $\Delta t = 8\tau$ or $\Delta t = 10\tau$, respectively. The maximal jamiton FD, constructed in Sect.~\ref{subsec:fundamental_diagram_maximal_jamiton}, is recovered in the limit $\alpha\to 0$, for which the averaging length goes to zero, and thus averages turn into point evaluations, i.e.
\begin{equation*}
\lim_{\Delta\eta\to 0}
\tfrac{1}{\Delta\eta}\int_{\bar{\eta}-\Delta\eta}^{\bar{\eta}} \rho(\eta)\ud{\eta}
= \rho(\bar{\eta})\;.
\end{equation*}
A comparison of the aggregated jamiton FDs with the data FD (shown in Fig.~\ref{fig:fd_data}) shows a better qualitative agreement than in the case without the aggregation. In particular, the very high values of $Q$, corresponding to high density parts of jamitons, are significantly reduced. An interesting feature that can be seen in Figs.~\ref{fig:fd_pw_uquad_plog_08},~\ref{fig:fd_pw_ukink_plog_10}, and~\ref{fig:fd_ar_ukink_pbddr_10} is the cusp at the $\rho_\text{R}$ point of the jamiton line that is parallel to the $\rho$ axis. The reason for this is that for $s=0$, the jamiton is stationary relative to the road, and thus aggregated values are identical to un-averaged point evaluations.

\subsection{Effective Flow Rate}
In all jamiton FDs shown in Figs.~\ref{fig:fd_pw_uquad_plog_00}--\ref{fig:fd_pw_uquad_plog_08}, Figs.~\ref{fig:fd_pw_ukink_plog_00}--\ref{fig:fd_pw_ukink_plog_10}, and Figs.~\ref{fig:fd_ar_ukink_pbddr_00}--\ref{fig:fd_ar_ukink_pbddr_10}, the jamiton region lies partially above the equilibrium curve. This is due to the fact that on their upstream side, jamitons possess densities that exceed the sonic density (see Fig.~\ref{fig:jamiton_curve_rho_u}). If averages involve more of this high-density part than of low-density parts, then the aggregated density and flow rate pair $(\bar{\rho},\bar{Q})$ may lie above the equilibrium curve.

To obtain the true effective flow rate that a chain of jamitons yields, one needs to perform the aggregation over infinitely long time intervals (we are neglecting the singularity at $s=0$ here), or equivalently, over full jamitons, i.e., from shock to shock. This allows one to investigate the important questions: 1) Do jamitons always lead to a reduction in the effective flow rate, compared to uniform traffic flow with the same average density? or 2) Can one construct physically realistic examples in which jamitons actually increase the effective flow rate?

To answer these questions, we first apply the full-jamiton averaging to the three test cases presented in Sect.~\ref{subsec:examples}. The resulting \emph{effective flow rate jamiton fundamental diagrams} are shown in Figs.~\ref{fig:fd_pw_uquad_plog_inf},~\ref{fig:fd_pw_ukink_plog_inf}, and~\ref{fig:fd_ar_ukink_pbddr_inf}. One can observe that in all three examples the jamiton region below the equilibrium curve remains unchanged, while it has fully vanished from above the equilibrium curve. The former observation is a consequence of the fact that on the one hand, very short jamitons possess average densities very close to their sonic density, and that on the other hand, very long jamitons possess average densities close to their minimum density. In contrast, the latter conclusion, which means that in the considered examples jamitons always reduce the effective flow rate, does not follow immediately from the shape of the jamitons. However, the theoretical answer to this question is provided by the following
\begin{thm}
\label{thm:effective_flow_rate}
The effective flow rate of a periodic chain of identical jamitons is always lower than the flow rate of a uniform base state solution with the same average density.
\end{thm}
\begin{proof}
The average flow rate $\bar{Q}$ being less than the flow rate of the uniform flow with the same average density, $Q(\bar{\rho})$, means that the averaged pair $(\bar{\rho},\bar{Q})$ lies below the equilibrium curve (see Fig.~\ref{fig:qrho_diagram}). Thus, we need to show that averages over full jamitons satisfy $\bar{\rho} < \rho_\text{S}$, where $\bar{\rho} = N/L$, and $N$ and $L$ are given by \eqref{eq:jamiton_length} and \eqref{eq:jamiton_mass}.

Due to the Rankine-Hugoniot shock conditions, given in Sect.~\ref{subsubsec:shocks}, a jamiton lies between two states that possess the same value of $r(\vv)$ (see Fig.~\ref{fig:rv_diagram_jamiton}). The function $r(\vv)$, given by \eqref{eq:definition_w_r} or \eqref{eq:definition_r}, respectively, can be inverted on two branches: let $\vv^+(r)$ denote its inverse function on $[\vv_\text{M},\vv_\text{S}]$, and $\vv^-(r)$ its inverse on $[\vv_\text{S},\vv_\text{R}]$. We can therefore parameterize the family of jamitons by $\mathring{r}\in [r(\vS),r(\vv_\text{R})]$. For each jamiton, we obtain
\begin{equation*}
\bar{\rho}(\mathring{r}) = N(\mathring{r})/L(\mathring{r})\;,
\quad\text{where}\quad
N(\mathring{r}) = \int_{\vv^+(\mathring{r})}^{\vv^-(\mathring{r})} \tfrac{r'(\vv)}{w(\vv)}\ud{\vv}
\quad\text{and}\quad
L(\mathring{r}) = \int_{\vv^+(\mathring{r})}^{\vv^-(\mathring{r})} \vv\tfrac{r'(\vv)}{w(\vv)}\ud{\vv}\;,
\end{equation*}
where the jamiton lies between the specific volumes $\vv^+(\mathring{r})$ and $\vv^-(\mathring{r})$.

We need to show that $N(\mathring{r})/L(\mathring{r}) \le \rho_\text{S}$, or $L(\mathring{r})/N(\mathring{r}) \ge \vS$, which is equivalent to the condition
\begin{equation*}
\int_{\vv^+(\mathring{r})}^{\vv^-(\mathring{r})} (\vv-\vS)\frac{r'(\vv)}{w(\vv)}\ud{\vv}\ge 0\;.
\end{equation*}
After changing the variable of integration to $\tilde{r} = r(\vv)$, this becomes
\begin{equation}
\label{eq:flow_rate_inequality_rnot}
\int_{r(\vS)}^{\tilde{r}}
\frac{\vv^-(\tilde{r})-\vS}{w(\vv^-(\tilde{r}))}-
\frac{\vv^+(\tilde{r})-\vS}{w(\vv^+(\tilde{r}))}
\ud{\tilde{r}}\ge 0\;,
\end{equation}
where the two quotients arise from the two branches of the inverse of $r(\vv)$. We now show that the integrand in \eqref{eq:flow_rate_inequality_rnot} is positive, which proves the claim.

As depicted in Fig.~\ref{fig:wv_diagram_secants}, consider two secants of the function $w(\vv)$: the first goes through the points $(\vv^+,w(\vv^+))$ and $(\vS,0)$, and the second goes through the points $(\vS,0)$ and $(\vv^-,w(\vv^-))$ (omitting the argument $\tilde{r}$ for notational economy). By assumption (see Sect.~\ref{subsec:assumptions_U_p_h}), $U'(\vv)>0$, thus both secant slopes are positive. Moreover, $U''(\vv)<0$, thus $w(\vv)$ is concave, and therefore, the slope of the first secant is larger than the slope of the second, i.e.
\begin{equation*}
\frac{w(\vv^+)}{\vv^+-\vS} > \frac{w(\vv^-)}{\vv^--\vS}\;.
\end{equation*}
Observing that the reciprocals of these slopes are the same as the two terms in the integrand of \eqref{eq:flow_rate_inequality_rnot} concludes the argument.
\end{proof}

\begin{figure}
\begin{minipage}[b]{.485\textwidth}
\includegraphics[width=\textwidth]{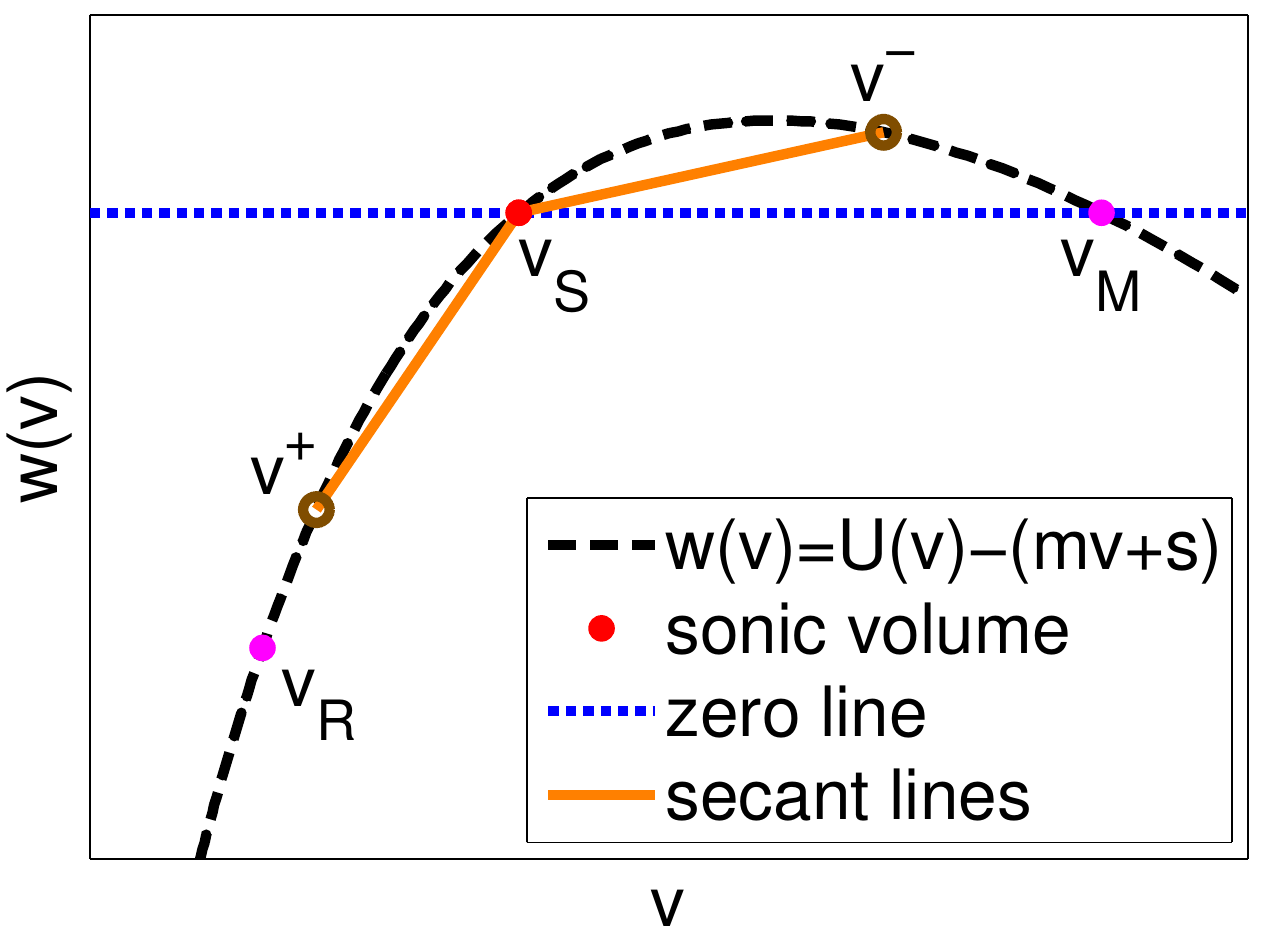}
\vspace{-1.8em}
\caption{Secants (orange lines) to the function $w(\vv)$ in the $u$--$\vv$ plane, connecting each of the jamiton volumes $\vv^+$ and $\vv^-$ with the sonic volume $\vS$.}
\label{fig:wv_diagram_secants}
\end{minipage}
\hfill
\begin{minipage}[b]{.485\textwidth}
\includegraphics[width=\textwidth]{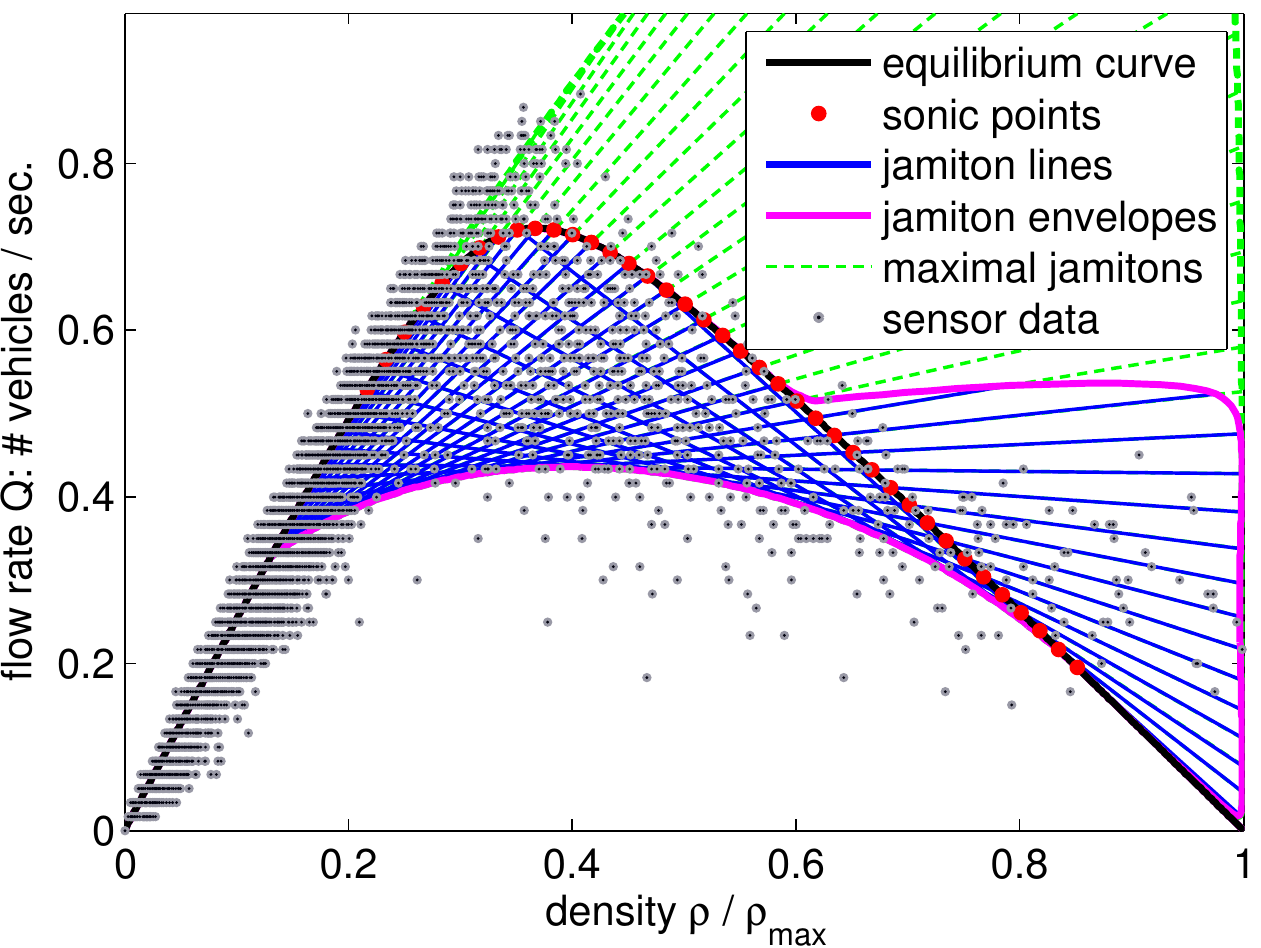}
\vspace{-1.8em}
\caption{Aggregated jamiton FD for the example \textbf{ARZ2} (see Sect.~\ref{subsec:examples}) in comparison with sensor data. \\}
\label{fig:fd_ar_with_data}
\end{minipage}
\end{figure}

\subsection{Mimicking Real Sensor Measurement Data}
\label{sec:mimicking_measurement_data}
The preceding investigations demonstrate that jamitons can serve as an explanation for multi-valued parts of a fundamental diagram. However, in reality a variety of additional effects are likely to contribute to the spread observed in FD data (see Fig.~\ref{fig:fd_data}), such as: variations in vehicle types and drivers, irreproducible behavior, lane switching, etc. In this section we investigate how well the spread in FD data can be reproduced by jamitons alone, while neglecting any other effects. To that end, we attempt to invert the jamiton FD construction: we consider the ARZ model \eqref{eq:aw_rascle_zhang_model} with parameters selected so as to illustrate that a qualitative agreement between the aggregated jamiton FD and the measurement data in Fig.~\ref{fig:fd_data} is possible.

The illustrative, rather than systematic, choice of parameters is due to the following fact. As the constructions in Sects.~\ref{subsec:fundamental_diagram_maximal_jamiton} and~\ref{subsec:emulating_sensor_measurements} indicate, the set-valued region in aggregated jamiton FDs (such as the ones shown in Figs.~\ref{fig:fd_ar_ukink_pbddr_01} and~\ref{fig:fd_ar_ukink_pbddr_10}) depends on both model functions: $U(\rho)$ and $h(\rho)$. Consequently, when observing a set-valued region in sensor measurements, one sees a combined result of both functions, and there is no canonical procedure to obtain them in succession (such as: first determine $U(\rho)$, and after that find $h(\rho)$).

We consider the example \textbf{ARZ2}, as described in Sect.~\ref{subsec:examples}, with the only modification that $u_\text{max} = 19.2\text{m}/\text{s}$. In addition, we choose $\alpha = 6$, i.e., for a data aggregation time of $\Delta t = 30\text{s}$, this corresponds to the drivers' relaxation time of $\tau = 5\text{s}$. While these model parameters are selected with the goal of achieving a qualitative agreement with data, none of them are visibly unreasonable or unrealistic. The results of this model can be seen in Fig.~\ref{fig:fd_ar_with_data}, which shows the aggregated jamiton FD, together with the measurement data. While the agreement is not perfect, many of the most significant features of the data are in fact reproduced by the jamiton FD.

\vspace{1.5em}
\section{Conclusions and Outlook}
\label{sec:conclusions_and_outlook}
We have demonstrated that both the classical Payne-Whitham model and the inhomogeneous Aw-Rascle-Zhang model of traffic flow possess an intrinsic phase transition behavior between stable and jamiton-dominated regimes. This structure gives rise to set-valued fundamental diagrams, and we have provided an almost fully explicit way to construct such set-valued fundamental diagrams from the parameters governing the second order traffic model. The set-valued part of the fundamental diagram represents families of traveling wave solutions, so-called \emph{jamitons}. Various ways to perform averaging on these solutions (thus emulating sensor measurements) lead to different types of envelopes for the set-valued region. Appropriate choices for the model parameter functions and for the type of averaging lead to fundamental diagrams that, at least qualitatively, agree with fundamental diagrams obtained from sensor measurements.

The phase transitions between stable and jamiton-dominated traffic flow are related to the sub-characteristic condition, that was introduced by Whitham in 1959, and the construction of jamiton solutions is based on the Zel'dovich-von~Neumann-D{\"o}ring (ZND) theory of detonations. Through both relations, interesting connections between traffic models and many other physical phenomena can be drawn (see \cite{FlynnKasimovNaveRosalesSeibold2009} for examples).

It is worthwhile to stress the essential difference of the relaxation models considered here with phase transition models (e.g., \cite{Colombo2002, Colombo2003, BlandinWorkGoatinPiccoliBayen2011}). While the latter insert the observed transition from function-valued to set-valued regions in the fundamental diagram explicitly into their models (i.e., different models are used for the different regions), the former consist of a single set of equations that generates the transition from function-valued to set-valued regions in a natural, intrinsic, fashion.

Having a systematic approach for the construction of the envelopes of a set-valued fundamental diagram from a given traffic model could also help in the inverse problem of determining desired velocity functions and traffic pressures from observations and data, as follows. While (at least for congested traffic) these parameter functions cannot easily be measured directly, envelopes of data points (modulo outliers) can certainly be obtained directly from measurements. Thus, the relationship between the fundamental diagram and jamiton solutions established in this paper could serve as means to re-construct parameter functions so that the model reproduces the observed envelopes.

The methodology presented here allows the complete construction of fundamental diagrams from second order traffic models. Yet, certain aspects require a deeper investigation. Most prominently, the dynamic stability of the jamiton solutions is not considered. While numerical studies (e.g., \cite{Greenberg2004, FlynnKasimovNaveRosalesSeibold2009}) suggest that many jamiton solutions are in fact stable attractors, this observation may not generalize to all jamiton solutions. The constructions conducted in Sect.~\ref{subsubsec:jamiton_construction} can, in principle, generate jamitons of arbitrarily short, and of arbitrarily long, wavelengths. Both are most likely dynamically unstable: a chain of very short jamitons is essentially a uniform density state plus some small sawtooth perturbation; and a very long jamiton possesses a very long ``tail'' of almost constant density. In both cases, if the corresponding density is itself unstable, small perturbations will amplify and thus destroy the structure. While it remains a topic of future research to analyze the structural stability of jamiton solutions, it can certainly be concluded that the removal of long wavelengths might result in interesting modifications of the resulting fundamental diagrams.

Numerical simulations can be employed to study the dynamic stability of jamiton solutions. In this fashion, fundamental diagrams can be systematically generated through numerical simulations (as conducted for instance in \cite{SiebelMauser2006}). Therefore, numerical studies could provide a missing link between fundamental diagrams obtained from sensor measurements and those constructed via jamiton solutions.

One limitation of this study is the restriction to the Payne-Whitham model and the Aw-Rascle-Zhang model only. Many more general types of second order traffic models do not fall into either of these two categories, in particular models that are partially based on measurement data (e.g., generalized Aw-Rascle-Zhang models \cite{FanHertySeibold2012}). In a companion paper \cite{KasimovRosalesSeiboldFlynn2012}, the results on the existence of jamitons are extended to a more general class of second order traffic models that contain both the PW pressure and the ARZ hesitation function at the same time.

\section*{Acknowledgments}
The authors would like to thank J.-C.\ Nave for helpful discussions. The authors would like to acknowledge the support by the National Science Foundation. R. R. Rosales and B. Seibold were supported through grants DMS--1007899 and DMS--1007967, respectively. In addition, R. R. Rosales wishes to acknowledge partial support by the NSF through grants DMS--0813648, DMS--1115278, and DMS--0907955, B. Seibold through grants DMS--0813648 and DMS--1115269, and A. R. Kasimov through grant DMS--0907955. M. R. Flynn wishes to acknowledge support by the NSERC Discovery Grant Program.

\bibliographystyle{plain}
\bibliography{references_complete}

\end{document}